\documentclass[12pt]{amsart}
\usepackage{graphicx}
\usepackage{amsmath, enumerate}
\usepackage{amsfonts, amssymb, amsthm, xcolor, stmaryrd}
\usepackage{tikz, tikz-cd, caption, tabu, longtable, mathrsfs, xtab, centernot, mathtools}
\usepackage{dsfont, cite}
\usepackage[]{todonotes}
\usepackage[pagebackref=true]{hyperref}

\numberwithin{equation}{subsection}
\theoremstyle{plain}

\newtheorem{thm}{Theorem}[section]
\newtheorem{lemma}[thm]{Lemma}
\newtheorem{prop}[thm]{Proposition}
\newtheorem{cor}[thm]{Corollary}

\theoremstyle{definition}

\newtheorem{conj}[thm]{Conjecture}
\newtheorem{exmp}[thm]{Example}

\theoremstyle{remark}
\newtheorem{rmk}[thm]{Remark}

\newcommand{\Hb}{\mathbb{H}}
\newcommand{\SL}{{\mathrm{SL}}}
\newcommand{\Zb}{\mathbb{Z}}
\newcommand{\Cb}{\mathbb{C}}
\newcommand{\Qb}{\mathbb{Q}}
\newcommand{\Nb}{\mathbb{N}}
\newcommand{\lp}{\left (}
\newcommand{\rp}{\right )}

\newcommand{\Gal}{{\mathrm{Gal}}}

\newcommand{\Rb}{\mathbb{R}}
\newcommand{\smat}[4]{\left(\begin{smallmatrix}
                 #1 & #2\\
                 #3 & #4
\end{smallmatrix}\right)}
\newcommand{\pmat}[4]{\begin{pmatrix}
                 #1 & #2\\
                 #3 & #4
\end{pmatrix}}

\newcommand{\half}{{\tfrac{1}{2}}}
\newcommand{\Ac}{{\mathcal{A}}}

\newcommand{\ebf}{{\mathbf{e}}}
\newcommand{\Hc}{{\mathcal{H}}}
\newcommand{\Hs}{{\mathscr{H}}}

\newcommand{\Dc}{{\mathcal{D}}}
\newcommand{\SO}{{\mathrm{SO}}}

\newcommand{\hT}{{\hat{\delta}}}

\newcommand{\tR}{\tilde{R}}
\newcommand{\tG}{\tilde{G}}

\newcommand{\Fc}{{\mathcal{F}}}
\newcommand{\Oc}{\mathcal{O}}
\newcommand{\df}{\mathfrak{d}}
\newcommand{\Cl}{{\mathrm{Cl}}}
\newcommand{\af}{\mathfrak{a}}
\newcommand{\Nm}{{\mathrm{Nm}}}

\newcommand{\Cc}{{\mathcal{C}}}
\newcommand{\Ec}{{\mathcal{E}}}

\renewcommand{\Mc}{{\mathcal{M}}}
\newcommand{\Sc}{{\mathcal{S}}}
\newcommand{\Pc}{{\mathcal{P}}}
\newcommand{\Lc}{{\mathcal{L}}}
\newcommand{\Nc}{{\mathcal{N}}}
\newcommand{\Uc}{{\mathcal{U}}}

\newcommand{\tr}{\mathrm{tr}}

\newcommand{\GSpin}{\mathrm{GSpin}}
\newcommand{\Res}{\mathrm{Res}}
\newcommand{\Wb}{\mathbb{W}}
\newcommand{\Ab}{\mathbb{A}}
\newcommand{\Db}{\mathbb{D}}

\newcommand{\D}{\mathrm{D}}
\newcommand{\E}{\mathrm{E}}

\newcommand{\Pb}{\mathbb{P}}
\newcommand{\Xb}{\mathbb{X}}
\newcommand{\Vb}{\mathbb{V}}

\newcommand{\rf}{\mathfrak{r}}
\newcommand{\tX}{{\tilde X}}
\newcommand{\tP}{{\tilde P}}
\newcommand{\trf}{{\tilde{\mathfrak{r}}}}
\newcommand{\tLc}{\tilde{ \mathcal{L}}}
\newcommand{\tU}{\tilde{U}}
\renewcommand{\th}{\tilde{h}}
\newcommand{\vtau}{{\tau}}

\newcommand{\tka}{{\tilde\kappa}}
\newcommand{\hka}{{\hat\kappa}}
 \newcommand{\vv}[1]{\overrightarrow{#1}}

\AtBeginDocument{%
   \def\MR#1{}
}
\begin{document}
\title{ Algebraicity of Higher Green Functions at a CM point}
\author[Yingkun Li]{Yingkun Li}
\address{Fachbereich Mathematik,
Technische Universit\"at Darmstadt, Schlossgartenstrasse 7, D--64289
Darmstadt, Germany}
\email{li@mathematik.tu-darmstadt.de}
\thanks{The author is supported by  the LOEWE research unit USAG,
    and by the Deutsche Forschungsgemeinschaft (DFG) through the Collaborative Research Centre TRR 326 ``Geometry and Arithmetic of Uniformized Structures'', project number 444845124.
}

\date{\today}
\maketitle
\begin{center}
{\it To Don Zagier on the occasion of his 70th birthday  }
\end{center}

\begin{abstract}

  In this paper, we investigate the algebraic nature of the value of a higher Green function on an orthogonal Shimura variety at a single CM point.
  This is     motivated by  a conjecture of Gross and Zagier in the setting of higher Green functions on the product of two modular curves.
  In the process, we will study analogue of harmonic Maass forms in the setting of Hilbert modular forms, and obtain results concerning the arithmetic of their holomorphic part Fourier coefficients. 
  As a consequence,  we answer a question of Zagier in his 1986 ICM proceeding.
\end{abstract}

\tableofcontents
\section{Introduction}

Let  $j(z)$ be the modular $j$-invariant on the modular curve $X_0(1) := \SL_2(\Zb) \backslash \Hb$ with
$\Hb$ the  upper-half plane. 
Its values at CM points
are algebraic integers called singular moduli. They play an important role in the explicit construction of class fields of imaginary quadratic fields. 

\subsection{Conjecture and Results}
The function $G_1(z_1, z_2) := 2 \log |j(z_1) - j(z_2)|$ is the automorphic Green function on $X_0(1) \times X_0(1)$, and the limiting member of a family of automorphic functions
\begin{equation}
  \label{eq:Gs}
  \begin{split}
    G_s(z_1, z_2) &:= -2 \sum_{\gamma \in \Gamma} Q_{s-1} \lp 1 + \frac{|z_1 - \gamma z_2|^2}{2 \Im(z_1) \Im(\gamma z_2)}\rp,~ \Re(s) > 1, \\
    Q_{s-1}(t) &:= \int^\infty_0 ( t + \sqrt{t^2 - 1} \cosh (u))^{-s} du,
  \end{split}
\end{equation}
that are eigenfunctions with respect to the Laplacians in $z_1$ and $z_2$.
For integral parameters $s = r + 1 \in \Nb$, these functions are called higher Green functions, and played an important role in calculating arithmetic intersections of Heegner cycles on Kuga-Sato varieties \cite{Zhang97}.
Given a weakly holomorphic modular form $f = \sum_{m \gg -\infty} c(m) q^m \in M^!_{-2r}$ on $X_0(1)$, one can 
associate to it a higher Green function
\begin{equation}
  \label{eq:Grf}
  \begin{split}
    G_{r+1, f}(z_1, z_2) &:= \sum_{m \in \Nb} c(-m) m^{r} G^m_{r+1}(z_1, z_2),\\
    G^m_{s}(z_1, z_2) &= G_s(z_1, z_2) \mid T_m = -2 \sum_{\gamma \in M_2(\Zb),~ \det(\gamma) = m} Q_{s-1} \lp 1 + \frac{|z_1 - \gamma z_2|^2}{2 \Im(z_1) \Im(\gamma z_2)}\rp.
  \end{split}
\end{equation}
Although the theory of complex multiplication does not directly apply, the values of $G_{r+1, f}$ at CM points on $X_0(1)^2$ should be algebraic in nature, as in the case of the automorphic Green function.
More precisely, these values should be algebraic multiples of logarithm of algebraic numbers.
This was conjectured in \cite[Conjecture (4.4)]{GZ86} when $z_1, z_2$ have the same discriminant, and mentioned as a question in \cite[section V.1]{GKZ87} for the general case (see also \cite{Mellit08} and \cite{Via11}).
In this paper, we prove the following result, which in particular solves problem (ii) raised by Don Zagier at the end of his 1986 ICM proceeding \cite{zagier-ICM}.

\begin{thm}
  \label{thm:GZ}
  Let $r \in \Nb$ and $f \in M^!_{-2r}$ with integral Fourier coefficients.
  Suppose $d_1, d_2$ are negative discriminants, such that one of them is fundamental when $r$ is odd.
  For any CM point $z_i$ with discriminant $d_i$, there exist $\kappa \in \Nb$ depending on $d_1, d_2, r$ and $f$, and $\alpha = \alpha(z_1, z_2) \in H$ such that
  \begin{equation}
    \label{eq:GZ}
    (d_1 d_2)^{r/2} G_{r+1, f}(z_1, z_2)  = \frac{1}{\kappa} \log |\alpha|.
  \end{equation}
  where $H = H_1H_2$ with $H_i$ the ring class field extension of $E_i := \Qb(\sqrt{d_i})$ associated to $z_i$.
  Furthermore, we have
  \begin{equation}
    \label{eq:Gal}
    \alpha(z_1^\sigma, z_2^\sigma) = \sigma(\alpha(z_1, z_2))
  \end{equation}
  for any $\sigma \in \Gal(H/E)$, where $E = E_1E_2$. 
\end{thm}
\begin{rmk}
  The group $\Gal(H/E)$ can be embedded as a subgroup of $ \Gal(H_1/E_1) \times\Gal(H_2/E_2)$, which then acts on the CM point $(z_1, z_2)$. 
\end{rmk}

There has been a lot of previous works concerning this question.
The first such result is due to Gross, Kohnen and Zagier \cite{GKZ87}, where $r$ is even\footnote{For odd $r$, they also obtain certain result, which turns out to be trivial in the case of level 1.}  
  and one considers average of the whole $\Gal(H/E)$-orbit of $(z_1, z_2)$.
In that case, the value is a rational multiple of the logarithm of a rational number.
When $E_1= E_2$
, this conjecture follows from the work of Zhang \cite{Zhang97}, under the assumption of the non-degeneracy of certain height pairing of Heegner cycles on Kuga-Sato varieties.
In \cite{Via11}, Viazovska gave an analytic proof without this assumption.
When $E_1 \neq E_2$
, Mellit \cite{Mellit08} gave a strategy to systematically verify this conjecture with one of the points fixed, and carried it out for $z_1 = i$. 
In \cite{Li18}, we considered the average over the whole $\Gal(H/E)$-orbit with $r$ odd, and were able to show that $\alpha \in \Qb(\sqrt{d_1d_2})$ and give an explicit factorization of the ideal it generates in the spirit of the seminal work of Gross and Zagier on singular moduli\cite{GZ85}.
Very recently, Bruinier, Ehlen and Yang made significant progress and proved algebraicity result in the sense of Theorem \ref{thm:GZ} by averaging over the Galois orbit of one of the two CM points with fundamental discriminant  \cite{BEY21}. 
We have now removed this averaging in Theorem \ref{thm:GZ} to obtain an algebraicity result at an individual CM point.

It is important to mention that one can replace $\SL_2(\Zb)$ with a congruence subgroup $\Gamma_0(N)$, define higher Green functions on $X_0(N)^2$ analogously, and ask the same question. This was in fact the setting that \cite{GKZ87} and \cite{Zhang97} were in. 
By viewing $X_0(N)^2$ as the Shimura variety for the $\Qb$-split group $\mathrm{O}(2, 2)$, it is natural to generalize the setting to an arbitrary  orthogonal Shimura variety, and ask the question about the algebraic nature of the CM-values of higher Green functions on such varieties.
This framework, which was adopted in \cite{BEY21}, will be the one we work in.

To be more precise, let $ V$ be a rational quadratic space of 
signature $(n, 2)$, and
\begin{equation}
  \label{eq:XKC}
  X_K(\Cb) = X_{V, K}(\Cb) = H(\Qb) \backslash \Db_V \times H(\hat \Zb)/K
\end{equation}
 the  Shimura variety associated to the algebraic group $H = H_V = \GSpin_V$ and an open compact subgroup $K \subset H(\hat\Zb)$ (see section \ref{subsec:QS} for details).
 Given an even, integral lattice $L \subset V$ such that $K$ fixes $\hat L = L \otimes \hat \Zb$  and acts trivially on the finite abelian group $\hat L'/\hat L$, one can associate a higher Green function    $\Phi_L^r(z, h, f)$ on $X_K$ to each weakly holomorphic modular form $f \in M^!_{1-n/2 -2r, \bar\rho_L}$ and $r \in \Nb$ (see Equation \eqref{eq:Phijf}).
 It has logarithmic singularity along special divisors on $X_K$. 

For   a totally real field $F$ of degree $d$, a quadratic CM extension $E/F$ becomes a binary $F$-quadratic space $W$ with respect to a quadratic form $\alpha \Nm_{E/F}$ for some $\alpha \in F^\times$. 
Suppose $W$ has signature $(( 0,2), ( 2,0), \dots, ( 2,0))$ with respect to the real embeddings $\sigma_1, \dots, \sigma_d$ of $F$ and there is an isometric embedding $W_\Qb := \Res_{F/\Qb}W \hookrightarrow V$.
This not only implies
\begin{equation}
  \label{eq:nd}
  n + 2 \ge 2d,
\end{equation}
but also gives CM points
$Z(W_\Qb) \subset X_K$
(see Equation \eqref{eq:ZW}).
This 0-cycle is defined over $F$, and each individual point  $(z_0, h) \in Z(W_\Qb)$ is defined over certain abelian extension of $E$. 
We will prove the following result concerning the algebraic nature of $\Phi^r_L$ at CM points in $Z(W_\Qb)$.
 \begin{thm}
   \label{thm:diff}
   In the setting above, suppose $f$ has integral Fourier coefficients.
   Then there exist
   algebraic numbers $\lambda_j \in F$   and $\alpha_j \in E^{\mathrm{ab}}$ for $1 \le j \le d$ such that 
  \begin{equation}
    \label{eq:maindiff}
    \Phi^r_L(z_0, h_1, f) -     \Phi^r_L(z_0, h_2, f) =
    \sum_{j = 1}^d \lambda_j \log \left| \frac{\sigma_{h_1} (\alpha_j)}{\sigma_{h_2}(\alpha_j)} \right|
  \end{equation}
  for any $(z_0, h_i) \in Z(W_\Qb)$.
Here $\sigma_h \in \Gal(E^{\mathrm{ab}}/E)$ is the element associated to $h \in   E^\times \backslash \hat E^\times$ via class field theory.
In particular when $F = \Qb(\sqrt{D})$ is real quadratic and $n +2 = 2d =4$, we can take $\lambda_1 = D^{r/2}/\kappa$ for some $\kappa \in \Nb$ and $\lambda_2 = 0$.
 \end{thm}

 \begin{rmk}
   \label{rmk:simple}
   When $(V, Q) = (M_2(\Qb), N \cdot \det)$, the Shimura variety $X_K(\Cb)$ becomes $X_0(N)^2$ for suitable $K$ \cite[section 3.1]{YY19}.
   In that case, for CM points $z_i$ with discriminant $d_i$, the CM point $(z_1, z_2)$ is in $Z(W_\Qb)$ with $W$ certain $F = \Qb(\sqrt{d_1 d_2})$-quadratic space. These are called ``big CM points'', resp.\ ``small CM points'', when $F$ is real quadratic, resp.\ $F = \Qb$.
 \end{rmk}

 \begin{rmk}
   \label{rmk:singularity}
Theorem \ref{thm:diff} applies    even when $Z(W_\Qb)$ intersects the singularity of $\Phi^r_L$, 
In that case, the function $\Phi^r(z, 1, f) - \Phi^r(z, h, f)$ in $z$ can be continued to a real-analytic function in the neighborhood of the singularity, and its value at $z = z_0$ defines the quantity on the left hand side of \eqref{eq:maindiff}. 
 \end{rmk}

 Previous results concerning any linear combinations of CM values of higher Green function either assume $n \le 2$ or $d = 1$. 
 Theorem \ref{thm:diff} is the first result where the cases with $n \ge 3$ and $d \ge 2$ are addressed.
In such cases, there is no known results even when one averages the higher Green function over all the CM points in all Galois conjugates of $Z(W_\Qb)$, unlike for Green functions studied in \cite{BY06, BKY12}.
 Together with  Theorem \ref{thm:GZ}, Theorem \ref{thm:diff}  naturally leads one to expect the following. 
 \begin{conj}
  \label{conj:main}
  In the setting of Theorem \ref{thm:diff}, suppose $f$ has integral Fourier coefficients and the singularity of $\Phi^r_L(z, h, f)$ does not intersect $Z(W_\Qb)$. Then there exists  $\lambda_j \in F$ and $\alpha_j \in E^{\mathrm{ab}}$ for $1 \le j \le d$ such that 
  \begin{equation}
    \label{eq:conj}
    \Phi^r_L(z_0, h, f) = \sum_{j = 1}^d \lambda_j \log \left| \sigma_h(\alpha_j) \right|
  \end{equation}
  for all $(z_0, h) \in Z(W_\Qb)$.
 \end{conj}
 When $F$ is real quadratic, i.e.\ $d = 2$, we can confirm it in the following case.
 \begin{thm}
   \label{thm:conj}
   Conjecture \ref{conj:main} holds when $F$ is a real quadratic field,  $r$ is even, $n = 4$, and $Z(W_\Qb)$ is defined over $\Qb$, in which case we can take $\lambda_1 \in \Qb$ and $\lambda_2 = 0$. 
 \end{thm}

 \begin{rmk}
   \label{rmk:defQ}
   When $E/\Qb$ is Galois, there are many instances when $Z(W_\Qb)$ is defined over $\Qb$ (see e.g.\ Lemma 3.4 in \cite{BY06}). In particular, the CM points on $X_0(1)^2$ satisfy this condition (see Example \ref{exmp:biquad}).
   Therefore, the case for even $r$ in Theorem \ref{thm:GZ} follows from Theorem \ref{thm:conj}.
 \end{rmk}

 \begin{rmk}
   In a recent joint work \cite{BLY22}, we have proved Conjecture \ref{conj:main} when $E/\Qb$ is biquadratic. 
 \end{rmk}
 \subsection{General Proof Strategy}
 \label{subsec:obs}
 When $F = \Qb$, Conjecture \ref{conj:main} follows from Theorem 5.5 in \cite{BEY21}.
Here we give a sketch of its proof, which is analytic in nature. 
First, one expresses $\Phi^r_L(z, h, f)$ as an integral of $f$ against a suitable theta kernel $R^r_\tau \Theta_L(\tau, z, h)$, where $R_\tau$ is the raising operator (see \eqref{eq:diffops}).
Then a CM point $(z_0, h) \in Z(W_\Qb)$ leads to a rational splitting of  $V$ since $F = \Qb$.
Suppose it leads to an integral splitting of $L$ into $L = \tilde L \oplus N$ with $\tilde L$ and $N$  definite lattices of signature $(n, 0)$ and $(0, 2)$ respectively. Then the theta kernel becomes \footnote{For simplicity, we omit the detail about the modular forms being vector-valued.}
\begin{equation}
  \label{eq:Tsplit}
\Theta_L(\tau, z_0, h) =    \theta_{\tilde L}(\tau) \theta_{N}(\tau),
\end{equation}
Note that $\theta_N$ is non-holomorphic and has weight $-1$. 
One can then construct a preimage $\hat\theta_N$ of $\theta_N$ under the lowering operator $L_\tau$. It is a harmonic Maass form of weight 1.
The notion of harmonic Maass form was introduced in the seminal work of Bruinier and Funke \cite{BF04}, and studied around the same time by Zwegers in the context of modular completion of Ramanujan's mock theta functions \cite{ZwThesis}. 

Applying the Rankin-Cohen operator (see \eqref{eq:RC}) to $\theta_{\tilde L}$ and $\tilde\theta_N$ then gives us a preimage of $R^r_\tau (\theta_{\tilde L}(\tau) \theta_N(\tau))$ under the $L_\tau$. 
Putting these together and applying Stokes' theorem gives us 
\begin{align*}
  \Phi^r_L(z_0, h, f)
&= \int_{X_0(1)} f(\tau) R^r_\tau( \theta_{\tilde L}(\tau) \theta_N(\tau)) d\mu(\tau)
  = \int_{X_0(1)} f(\tau) L_\tau \Cc_{(n/2, 1), r}( \theta_{\tilde L}(\tau),  \tilde\theta_N(\tau)) d\mu(\tau)\\
&= \{f(\tau),  \Cc_{(n/2, 1), r}( \theta_{\tilde L}(\tau),  \tilde\theta^+_N(\tau))\}
= \mathrm{CT}(\tilde f (\tau) \cdot \tilde\theta_N^+(\tau)),
\end{align*}
where $\{,\}$ is a pairing of formal Fourier series (see \eqref{eq:pairing}) and $\mathrm{CT}$ denotes the constant term of a Fourier series.
The function $\tilde\theta^+_N$ is the holomorphic part of $\tilde\theta_N$, and the modular form $\tilde f$ in the last expression is weakly holomorphic with weight $-1$ and rational Fourier coefficients. 

The harmonic Maass form $\tilde\theta_N$ of weight 1 was studied in \cite{DL15, Ehlen17, Viazovska19}.
It was shown that the term $ \mathrm{CT}(\tilde f (\tau) \cdot \tilde\theta_N^+(\tau))$ is the logarithm of an algebraic number.
To see this, let $P_1, P_2$ be positive definite, unimodular lattices such that $\theta_{P_1} - \theta_{P_2}$ is holomorphic on $\Hb$. 
One can rewind the process above (with $r = 0$) and write
\begin{equation*}
  \begin{split}
  \mathrm{CT}(\tilde f (\tau) \cdot \tilde\theta_N^+(\tau))
  &= \int_{X_0(1)} \tilde f(\tau) \theta_N(\tau) d\mu(\tau)
    =  \int_{X_0(1)} \hat f(\tau)
    (\theta_{P_1}(\tau) - \theta_{P_2}(\tau))
    \theta_N(\tau) d\mu(\tau)\\
&= \Phi_{L_1}(z_1, \hat f) -     \Phi_{L_2}(z_2, \hat f),
\end{split}
\end{equation*}
where $\hat f = \tilde f \cdot (\theta_{P_1} - \theta_{P_2})^{-1}$ is weakly holomorphic and $L_i = P_i \oplus N$.
The functions $\Phi_{L_i}(z, \hat f)$ are the regularized Borcherds lifts of $\hat f$ and are logarithms of rational functions on Shimura varieties associated to $L_i$.
Their values at CM points $z_i$ are logarithms of algebraic numbers by the theory of complex multiplication.
This finishes the sketch of the proof. 
The process of multiplying and dividing by $\theta_{P_1} - \theta_{P_2}$ a manifestation of the embedding trick (see \cite[section 8]{Borcherds98}).

The partial averaging result in \cite[Theorem 1.2]{BEY21} used the coincidence that the average of $G_{r+1, f}(z_1, z)$ over the Galois orbit of $z_1$ is a higher Green function in $z$ on the modular curve, i.e.\ $n +2 = 3$.
This is a rather special phenomenon that only happens when $E/\Qb$ is biquadratic. 
By \eqref{eq:nd}, one is reduced to the case of $d = 1$ in Conjecture \ref{conj:main}.

For  $d \ge 2$, the lattice $L$ splits as $\tilde L \oplus \Res_{F/\Qb} N$ with $N \subset W$ an $\Oc$-lattice of signature $(( 0,2), (2 , 0)\dots, (2, 0))$, and the analogue of \eqref{eq:Tsplit} is
\begin{equation}
  \label{eq:Tsplit1}
\Theta_L(\tau, z_0, h) =    \theta_{\tilde L}(\tau) \theta^\Delta_{N}(\tau),
\end{equation}
where $\theta^\Delta_N(\tau)$ is the diagonal restriction of the Hilbert theta function $\theta_N(\tau_1, \dots, \tau_d)$ of weight $(-1, 1 \dots, 1)$ associated to $N$.
When one tries to execute the above strategy
to construct a preimage of $\theta^\Delta_N$ under the lowering operator, it is necessary to work with Hilbert modular forms, and there are some serious obstacles.

\begin{itemize}
\item 
  The analogous $\hat\theta_N$  should be a Hilbert modular form that
satisfies suitable properties similar to those of harmonic Maass forms in the setting of elliptic modular forms. 
However, there is no suitable extension of the notion of harmonic Maass forms to higher rank groups.

\item
As harmonic Maass forms have singularities at the cusps, one would expect the same for the analogous $\hat\theta_N$. 
However,  Koecher's principle would imply that such $\hat\theta_N$ could not have singularity only at the cusps, but in the interior of the Hilbert modular variety as well. This also holds for its diagonal restriction and complicates the application of Stokes' theorem.

\item
  To extract information about the Fourier coefficient of $\hat\theta_N$, one needs the generalization of Borcherds' lift over totally real fields.
In a large part, this has been accomplished in \cite{Br12} by considering regularized theta lifts of Whittaker forms. However, as the Shimura varieties appeared loc.\ cit.\ are compact, there is no Fourier expansion and one has limited information about the rationality of the lift. 
Furthermore,  it seems hopeless to direct generalize the embedding trick in \cite{Borcherds98} to  totally real fields.
\end{itemize}

Instead of studying the value at an individual CM point, one can average over CM points in $Z(W_\Qb)$, and those in $Z(W(j)_\Qb)$ for $2 \le j \le d$, where each $W(j)$ is a neighboring $F$-quadratic space of $W$ (see section \ref{subsec:QS}).
Then the rational quadratic spaces $\Res_{F/\Qb} W(j)$ are all isomorphic and
\begin{align*}
\sum_{1 \le j \le d} \sum_{(z, h) \in Z(W(j)_\Qb)} \Theta_L(\tau, z, h) &=
\theta_{\tilde L}(\tau) \otimes
\sum_{1 \le j \le d}E^\Delta_{N(j)}(\tau)
,\\
E_{N(j)}(\tau_1, \dots, \tau_d) &:= \sum_{(z_j, h(j)) \in Z(W(j)_\Qb)} \theta_{N(j)}((\tau_1, \dots, \tau_d), h(j))  ,
\end{align*}
with $N(j) \subset W(j)$ suitable lattices.
The Hilbert modular form $\theta_{N(j_0)}$ is holomorphic in $\tau_j$ for $j \neq j_0$ and has weight $(1, \dots, 1, -1, 1, \dots, 1)$ with $-1$ at the $j_0$-th place.
One can now explicitly construct an incoherent Hilbert Eisenstein series $\Ec_N$ of parallel weight 1 that maps to $E_{N(j)}$ under the lowering operator in $\tau_j$ for all $1 \le j \le d$.
For $d = 2$, this is the real-analytic Eisenstein series that appeared in the seminal works of Gross and Zagier on singular moduli and the Gross-Zagier formula \cite{GZ86, GZ85}.
It also appeared in \cite{BY06}, and has been combined with the regularized theta lifting of Borcherds to give fruitful generalizations of \cite{GZ86, GZ85} in \cite{BY09, BKY12}.

The advantage of $\Ec_N$ is that its Fourier coefficients can be computed explicitly, and shown to be logarithms of rational numbers. They furthermore can be interpreted as arithmetic intersection numbers.
On the other hand, it provides limited information about the arithmetic of higher Green function at a single CM point, as the differential operator in the strategy for $d = 1$ does not readily generalize except in the case $n = 2d = 4$ and $r$ even (see the discussion at the end of section 5 in \cite{BEY21}).
The higher Green functions studied by Gross, Kohnen and Zagier in  \cite{GKZ87} happen to be in this single case.

\subsection{Ideas}
We now describe some ideas and observations that help to overcome the obstacles mentioned in the previous section:
\begin{itemize}

\item
  For any holomorphic Hilbert cusp form $g(\tau)$ and $\ell \in \Nb$ sufficiently large, the product $g(\tau)^\ell\theta_N(\tau)$ has a modular preimage under the lowering operator in $\tau_1$ with no singularity in $\Hb^d$. Furthermore, this preimage is harmonic in $\tau_1$ and holomorphic in $\tau_2, \dots, \tau_d$. 
\item
  The generalization of Borcherds' lift in \cite{Br12} differs from the logarithm of an $F$-rational function by a locally constant function, which can be canceled out when considering differences of linear combinations of CM values.
\item
  The embedding trick only needs to work along the diagonal of $\Hb^d$, and one can apply the Siegel-Weil formula to replace the difficult task  of finding suitable positive definite $\Oc$-lattice $P$ to the simpler one of analyzing Eisenstein series.
\end{itemize}

The first idea is inspired by Zwegers' work \cite{ZwThesis}, where the product of a mock theta function and a classical theta function is completed to become a real-analytic modular form without singularity in $\Hb$.
Such products are also called ``mixed mock-modular forms'' in \cite{DMZ12} and are natural objects to consider.
Since the differential operators in $\tau_1, \dots, \tau_d$ are all independent, this idea can be applied in the setting of Hilbert modular forms.
The existence of the modular preimage will be proved using complex geometry (see section \ref{subsec:exist}), as done in the elliptic case in \cite{BF04}. 
The parameter $\ell$ serves to ensure certain cohomology group vanishes (see Theorem \ref{thm:vanish}).
The Rankin-Cohen differential operator can also be generalized to be applied on such functions (see the differential operator $\Dc_{\kappa, r}$ in \eqref{eq:Dc}).

The second idea is a compromise so that one can still use the generalization of Borcherds' lift in \cite{Br12} to deduce algebraicity results.
Considering differences is quite effective in removing the so-called ``normalizing constant'' in the regularized theta lift (see Theorem 1.1 in \cite{BY11}), as different linear combinations could give rise to the same normalizing constant.
Furthermore, considering the difference turns out to simplify many other situations as well. For example, it is enough to construct a preimage of $g(\tau)^\ell (\theta_{N_1}(\tau) - \theta_{N_2}(\tau))$ with $g(\tau)$ a holomorphic Hilbert cusp form.
This is accomplished in Theorem \ref{thm:L1pre}, using the ampleness of twists of determinant of the Hodge bundle on toroidal compactifications of Hilbert modular varieties, which is contained in Theorem \ref{thm:vanish} and a result of independent interest
\footnote{We thank the referee for a helpful suggestion that led to this result.}.
Also, one does not need to worry so much about the singularity of $\Phi^r_L$ (see Remark \ref{rmk:singularity} and Lemma \ref{lemma:Phidiff}).
The linear combination we take will come from multiplying this preimage with an Eisenstein series $E_{\tP}$. 
This leads to the crucial algebraicity result in Theorem \ref{thm:algebraicity}, which is of independent interest.

For the embedding trick, the last idea reduces the problem of dividing by a Hilbert cusp form $g$, which is constructed from theta series, to dividing by its diagonal restriction $g^\Delta$, which is an elliptic modular form.
Using the Siegel-Weil formula, we can relate $g$ to Hilbert Eisenstein series. 
By varying the weight, we will show that for any finite set of points in $\Hb^d$, there is a Hilbert Eisenstein series that does not vanish on this set (see Lemma \ref{lemma:supp}). This observation has its root in the classical work \cite{RSD} of Rankin and Swinnerton-Dyer  on zeros of elliptic Eisenstein series, and leads to the ``partition of unity'' result in Proposition \ref{prop:POU}.

By putting these ideas together, we are able to overcome the obstacles and prove Theorem \ref{thm:diff}.
When it is specialized to the case in Theorem \ref{thm:GZ}, we can combine this result about differences with the  result about partial averages in \cite{BEY21} to complete the proof.

\subsection{Outlook and Organization}
To prove Conjecture \ref{conj:main}, one needs algebraicity results concerning sums of CM points, in addition to the ``difference result'' in  Theorem \ref{thm:diff}.
For real quadratic $F$, we have worked out such a ``sum result'' when $E/\Qb$ is biquadratic in \cite{BLY22}, which has led to a proof of Conjecture \ref{conj:main} in this case.
When $d \ge 3$, one can try to relate the (in)coherent Eisenstein series to Eisenstein series on $\mathrm{O}(2, 1)$ over $F$, and realize them as suitable theta lifts from $\SL_2$ over $F$. 
 We plan to pursue this idea in a future work.

 The paper is organized as follows.
 In sections \ref{sec:prelim} and \ref{sec:func}, we setup notations and collect various preliminary notions from the literature.
Results such as Lemma \ref{lemma:supp} and Proposition \ref{prop:POU} seem to be new, and form a crucial trick in the proof of Theorem \ref{thm:diff}.
In section \ref{sec:mixmock}, we construct certain real-analytic Hilbert modular form in Theorem \ref{thm:L1pre} and prove algebraicity result about linear combinations of their Fourier coefficients in Theorem \ref{thm:algebraicity}.
Putting these together, we give the proofs of Theorems \ref{thm:GZ}, \ref{thm:diff} and \ref{thm:conj} in section \ref{sec:pfs}.

\textbf{Acknowledgement}: We thank Jan H.\ Bruinier, Ben Howard, Steve Kudla, Jakob Stix Torsten Wedhorn, and Tonghai Yang for helpful discussions. 
We also appreciate fruitful conversations with Stephan Ehlen and Maryna Viazovska over the years concerning CM values of higher Green functions.
We are thankful for the thorough reading and helpful comments by the anonymous referee. 
Finally, we thank Don Zagier for drawing our attention to this problem and many encouraging discussions.
 
\section{Preliminary}
\label{sec:prelim}

Fix an embedding $\overline{\Qb} \hookrightarrow \Cb$. 
Throughout the paper, $F$ will be a totally real field of degree $d$ with ring of integers $\Oc$, different $\df$ and discriminant $D$. For $1 \le j \le d$ and $m \in F$, denote $\sigma_j: F \hookrightarrow \Rb$ the real embeddings of $F$ and $m_j := \sigma_j(m)$.
We write $m \gg 0$ if $m \in F$ is totally positive, i.e.\ $m_j > 0$ for all $1 \le j \le d$.
For a number field $E$ with ring of integers $\Oc_E$, let $\Ab_E$ and $\hat E := E \otimes \hat \Zb$ be the adeles and finite adeles respectively.
The subgroup $\hat\Oc_E := \Oc_E \otimes \hat\Zb \subset \hat E$ is open and compact.

Given $\vtau = (\tau_1, \dots, \tau_d) \in \Hb^d$, we write $v = (v_1, \dots, v_d) := \Im(\tau) \in (\Rb_{>0})^d$.
For a function $f$ on $\Hb^d$, we will write $f^\Delta$ for its diagonal restriction to $\Hb \subset \Hb^d$. 
For $\alpha \in \Cb$, denote
$$
\ebf(\alpha) := e^{2\pi i \alpha}.
$$
For $x = (x_j)_{1 \le j \le d}, y = (y_j)_{1 \le j \le d} \in \Cb^d$, we denote
$$
xy = (x_j y_j)_{1 \le j \le d} \in \Cb^d,~
\tr(x) := \sum_{1 \le j \le d} x_j,~
\Nm(x) := \prod_{1 \le j \le d} x_j.
$$
For a semigroup $G$ and a $G$-graded ring $R = \oplus_{i \in G} R_i$, we use
$$
R^{G_0} := \bigoplus_{i \in G_0} R_i \subset R
$$
for a sub-semigroup $G_0 \subset G$.
Also, we denote $\Nb$ the positive integers and $\Nb_0 := \Nb \cup \{0\}$.

\subsection{Modular Forms}
For a congruence subgroup $\Gamma \subset \SL_2(\Oc)$, a finite dimensional, unitary representation $\rho$ of $\Gamma$ on a finite dimensional hermitian space $(\mathcal{V}, \langle\cdot, \cdot \rangle)$, and weight $\kappa = (k_1, \dots, k_d) \in \Zb^d$, let $\Ac^{}_{\kappa, \rho}(\Gamma)$ denote the $\Cb$-vector space of vector-valued, real-analytic functions
on $\Hb^d$ invariant with respect to $\rho$ on $\Gamma$ of weight $\kappa$, and bounded near the cusps of $\Gamma \backslash \Hb^d$.
It contains the subspaces $S_{\kappa, \rho}(\Gamma) \subset M_{\kappa, \rho}(\Gamma)$ of cuspidal and holomorphic Hilbert modular forms.
We also write $\vv{k} := (k ,\dots, k)$.

For any $f \in \Ac_{\kappa, \rho}(\Gamma)$, the function on $\Hb^d$
\begin{equation}
  \label{eq:Pet0}
  \| f(\tau) \|_{\mathrm{Pet}}^2  := \langle f(\tau), { f(\tau)} \rangle \Nm(v^{\kappa})
\end{equation}
is $\Gamma$-invariant.
Given $f, g \in \Ac_{\kappa, \rho}(\Gamma)$ such that at least one of them has exponential decay near the cusps, we can define their Petersson inner product
\begin{equation}
  \label{eq:Pet}
  (f, g)_{\mathrm{Pet}} := \frac{1}{\sqrt{D}} \int_{\Gamma \backslash \Hb^d} \langle f(\tau), { g(\tau)} \rangle \Nm(v^\kappa) d\mu(\tau),~ 
  \end{equation}
  where $d\mu(\tau) := d\mu(\tau_1)\dots d\mu(\tau_d)$ is the invariant measure on $\Hb^d$ (see Equation (4.21) in \cite{Br12}).
For   $\kappa = (k_1, \dots, k_d)$, denote the following related weights
\begin{equation}
  \label{eq:hatkappa}
  \tka := (-k_1, k_2, \dots, k_d),~
  \hka := (2 - k_1, k_2, \dots, k_d).
\end{equation}
We omit $\Gamma$, resp.\ $\rho$, from the notation when $F$ is fixed and $\Gamma = \Gamma_F := \SL_2(\Oc)$, resp.\ it is trivial.
When $F = \Qb$, we will use the superscript $!$ to indicate modular forms with singularities at the cusps. 

When $\rho$ is trivial, it is known
that $M_{\kappa}(\Gamma_F) = M_\kappa(\Gamma_F, \Qb) \otimes \Cb$, where $M_\kappa(\Gamma_F, \Qb)$ is the subspace of modular forms with rational Fourier coefficients.
This is also the case for $M_{\kappa, \rho}$ when $\rho$ is a Weil representation defined below (see \cite{McGraw03}, \cite[section 7]{Br12}).

For later purposes, we will be interested in the ($\Nb^d$-)graded ring
\begin{equation}
  \label{eq:Mc}
  \Mc_F := \bigoplus_{\kappa \in \Nb^d} M_\kappa(\Gamma_F, \Qb). 
\end{equation}

\subsection{Differential Operators}
For $k \in \Zb$, we have the usual raising, lowering and hyperbolic Laplacian operators on $\Hb$
\begin{equation}
\label{eq:diffops}
\begin{split}
  R_{\tau, k} &:= 2i \partial_\tau + \frac{k}{v},~
  \tR_{\tau, k} := (4\pi)^{-1} R_{\tau, k},~ 
L_{\tau, k} := -2iv^2 \partial_{\overline{\tau}}, \\
\Delta_{\tau, k} &:= - R_{\tau, k - 2} L_{\tau, k} = - L_{\tau, k + 2} R_{\tau, k} - k = -v^2 \lp \partial_u^2 + \partial_v^2 \rp + ik v (\partial_u + i \partial_v).
\end{split}
\end{equation}
They change the weight by $2, -2$ and 0 respectively.
For $\kappa = (k_1, k_2) \in  \Qb^2 $ and $r \in \Nb_0$, we can define the Rankin-Cohen operator on a real-analytic function $f(\tau_1, \tau_2) \in \Hb^2$ by
\begin{equation}
  \label{eq:RC}
  \begin{split}
    \Cc_{\kappa, r}(f)(\tau) &:= (2\pi i)^{-r} \sum^r_{s = 0} (-1)^s \binom{k_1 + r - 1}{s} \binom{k_2 + r - 1}{r - s}    \lp \partial_{\tau_1}^{(r-s)} \partial_{\tau_2}^{(s)} f \rp(\tau, \tau), \\
    &=  \sum^r_{s = 0} (-1)^{r-s} \binom{k_1 + r - 1}{s} \binom{k_2 + r - 1}{r - s}
    \lp \tR^{r-s}_{\tau_1, k_1}  \tR^s_{\tau_2, k_2} f\rp (\tau, \tau),
  \end{split}
\end{equation}
where
$\binom{m}{n} := \frac{m(m-1)(m-2)\dots(m-n + 1)}{n!}$ is the binomial coefficient.
The equality in the second line can be proved by considering the generating series constructed from the differential operators $\partial_\tau$ and $R_\tau$. The details are contained in section 5.2 of \cite{BGZ08}, in particular Propositions 18 and 19.
The first expression shows that the operator preserves holomorphicity. 
When $\kappa \in \Zb^2$, the second expression shows that it preserves modularity in the sense that
\begin{equation}
  \label{eq:RCmod}
  \Cc_{\kappa, r}(f) \mid_{k_1 + k_2 + 2r} \gamma =
    \Cc_{\kappa, r}(f \mid_{\kappa} (\gamma, \gamma))
\end{equation}
for any $\gamma \in \SL_2(\Rb)$.
The same result holds in the metaplectic setting when $\kappa \in \half \Zb^2$.
\begin{exmp}
  \label{exmp:wt11}
  Suppose $\kappa = (1, 1)$ and $f(\tau_1, \tau_2) = \ebf(\alpha_1 \tau_1 + \alpha_2 \tau_2)$ for $\alpha \in \Qb(\sqrt{D})$ with $\tr(\alpha) \neq 0$. Then
  \begin{align*}
\ebf(-\tr(\alpha)\tau)
&    \Cc_{\kappa, r}(f)(\tau) =
                             \sum^r_{s = 0}  \binom{r}{s}^2     \alpha_1^{r-s} (-\alpha_2)^{s}
                               =
                               \tr(\alpha)^r2^{-r}
                             \sum^r_{s = 0} \binom{r}{s}^2     \lp \frac{2\alpha_1}{\alpha_1 + \alpha_2}\rp^{r-s} \lp \frac{- 2\alpha_2}{\alpha_1 + \alpha_2}\rp^{s}\\
    &=
                               \tr(\alpha)^r2^{-r}
      \sum^r_{s = 0} \binom{r}{s}^2
(x + 1)^{r-s} (x-1)^{s}
      =
\tr(\alpha)^r
      P_r (x)
  \end{align*}
  with
$ x =     \frac{\alpha_1 - \alpha_2}{\alpha_1 + \alpha_2} =   \frac{\sqrt{D}\tr(\alpha/\sqrt{D})}{\tr(\alpha)} $
and
$P_r(X)$  the $r$-th Legendre polynomial, which has parity $(-1)^r$.
The last equality is a consequence of Rodrigues' formula (see (8.6.18) in \cite{AS64}).
This example will be used in the proof of Theorem \ref{thm:diff}.
\end{exmp}

When $f(\tau_1, \tau_2) = f_1(\tau_1)f_2(\tau_2)$ with $f_i$ modular forms of weight $k_i \in \half\Zb$, the function $[f_1, f_2]_r := \Cc_{\kappa, r}(f)$ is the usual Rankin-Cohen bracket of $f_1$ and $f_2$ \cite[section 5.2]{BGZ08}.
If $f_1$ is harmonic of weight $k_1$ and $f_2$ is holomorphic of weight $k_2$, then we have
\begin{equation}
  \label{eq:LRC}
  L_\tau [f_1, f_2]_r
  = \binom{k_1 +r -1}{r}
    L_\tau f_1 \tR^r_{k_2} f_2
  = \binom{k_1 +r -1}{r} \tR^r_{k_1 + k_2-2}L_\tau (f_1f_2)
\end{equation}
for any $r \in \Nb_0$. 
\begin{lemma}
  \label{lemma:RRC}
  For real-analytic functions $f, g: \Hb \to \Cb$, rational numbers $k, \ell$ and an integer $r \ge 0$,
  if $ k + \ell \not \in \{-2r+2, -2r+3, \dots, 0\}$\footnote{We take this set to be empty for $r =0 $.},
  then there exists  $c_{r, a, j} \in \Qb$ such that
  \begin{equation}
    \label{eq:comb}
     \tR^a_k f \tR^{r-a}_\ell g =  \sum_{j = 0}^r c_{r, a, j} \tR^{r-j}_{k + \ell + 2j} [f, g]_j
  \end{equation}
  for all $0 \le a \le r$. 
\end{lemma}

\begin{proof}
  This is done by induction on $r$. The base case of $r = 0$ is trivial. For the inductive step to prove the case $r+1$, we have $k+\ell \not\in\{-2r, -2r+1, \dots, 0\}$.
  Denote $x_a := \tR^a f \tR^{r+1-a}g$ for $0 \le a \le r+1$.
  Applying $\tR$ to \eqref{eq:comb} shows that  $x_a + x_{a+1}$  is a rational linear combination of $\tR^{r+1-j}[f, g]_j$'s over $0 \le j \le r$ for any $0 \le a \le r$.
  From definition, we also have
  $$
  \sum^{r+1}_{a = 0}
  c_a  x_{a}  =
  [f, g]_{r+1},~ c_a :=  (-1)^a \binom{k + r}{r+1-a} \binom{\ell + r }{a}. 
  $$
  Therefore, the right hand side below is a rational linear combination of $\tR^{r+1-j}[f, g]_j$ for $0 \le j \le r+1$ and it suffices to show that the square matrix on the left is invertible
  $$
  \binom{A}{c_0 ~ c_1 \dots c_{r} ~ c_{r+1}}
    \cdot
  \begin{pmatrix}
    x_0\\
    \vdots\\
    x_{r+1}
  \end{pmatrix}
  =
  \begin{pmatrix}
    x_0 + x_1 \\
    \vdots \\
    x_{r} + x_{r+1}\\
    [f, g]_{r+1}
  \end{pmatrix},~
  A :=
  \begin{pmatrix}
    1 ~ 1 ~ 0 ~ 0 \dots 0\\
    0 ~ 1 ~ 1 ~ 0  \dots 0\\
    \dots  \\
    0 ~ 0 \dots 0 ~ 1 ~ 1
  \end{pmatrix}.
  $$
  The right kernel of $A$ is spanned by the vector $((-1)^a)_{0 \le a \le r+1}$.
  On the other hand
$$
\sum^{r+1}_{a = 0} c_a\cdot (-1)^a 
=
\sum^{r+1}_{a = 0} 
\binom{k + r}{r+1-a} \binom{\ell + r }{a} =  \binom{k + \ell + 2r}{r+1}, 
$$
which is zero precisely when $k + \ell \in \{-2r, -2r+1, \dots, -r\}$.
This is not possible by the condition imposed on $k+\ell$. Therefore the matrix $\binom{A}{c_0 \dots c_{r+1}}$ is invertible.
\end{proof}
Now we will extend the Rankin-Cohen operator to functions on $\Hb^d$ for any $d \ge 2$ by first restricting it to $\Hb^2$, before applying the usual Rankin-Cohen operator.
This can be expressed as a linear combination of the generalized Rankin-Cohen operators studied in \cite{Lee04}.
For $f: \Hb^d \to \Cb$, $\kappa = (k_1, \dots, k_d) \in \Zb^d$, denote
$f^{\Delta, 1}(\tau', \tau_1) := f(\tau_1, \tau', \dots, \tau')$, $\kappa(1) := (\tr(\kappa) - k_1, k_1) \in \Zb^2$ and define 
\begin{equation}
  \label{eq:RCgen}
  \Cc^1_{\kappa, r}(f)(\tau) := \Cc_{\kappa(1), r}(f^{\Delta, 1})(\tau).
\end{equation}
It is easy to check that $  (  f \mid_\kappa (\gamma, \dots, \gamma))^{\Delta, 1} =
    f^{\Delta, 1} \mid_{\kappa(1)} (\gamma, \gamma)$ and 
\begin{equation}
  \label{eq:RCmodj}
  \Cc^1_{\kappa, r}(f)\mid_{\tr(\kappa) + 2r} \gamma = \Cc^1_{\kappa, r}(f \mid_\kappa (\gamma, \dots, \gamma))
\end{equation}
for all $\gamma \in \SL_2(\Rb)$.
Suppose $f(\tau) = q_1^{\alpha_1} \dots q_d^{\alpha_d}$ with $q_j := \ebf(\tau_j),~ \alpha_j \in \Cb$, then
\begin{equation}
  \label{eq:RCq}
  \Cc^1_{\kappa, r}(f) = \sum_{s = 0}^r (-1)^s
  \binom{k_1 + r - 1}{s} \binom{\tr(\kappa) - k_1 + r - 1}{r - s}
  \alpha_1^{r-s} (\alpha - \alpha_1)^{s}  q^{\alpha},
\end{equation}
where $\alpha := \alpha_1 + \dots + \alpha_d$. 
Analogous definitions also make sense when the index 1 above is replaced by any $j \in \{1, \dots, d\}$.

Finally for $f, g \in \Hb^d \to \Cb$ real-analytic and $\kappa \in \Zb^d$, we define
\begin{equation}
  \label{eq:Dc}
  \Dc_{\kappa, r}(f, g) := (g^{r+1})^\Delta \Cc^1_{\kappa, r}(f/g)
  = (2\pi i)^{-r} \lp \sum_{\substack{e = (e_1, \dots, e_d) \in \Nb_0^d\\ \tr(e) = r}}
  a_e g^{r+1} \partial_{\tau_1}^{e_1} \dots \partial_{\tau_d}^{e_d} (f/g) \rp^\Delta
\end{equation}
with $a_e \in \Zb$ explicit constants given by
$$
a_e :=  (-1)^{e_2 + \dots + e_d} \binom{k_1 + r - 1}{e_2 + \dots + e_d} \binom{k_2 + \dots + k_d + r - 1}{e_1} \frac{s!}{e_2! \dots e_d!}.
$$
From the definition, one sees that $\Dc_{\kappa, r}(f, g)$ is real-analytic on $\Hb$ and satisfies
\begin{equation}
  \label{eq:Dcmod}
\Dc_{\kappa, r}(f\mid_{\kappa + \lambda}(\gamma, \dots, \gamma), g\mid_{\lambda} (\gamma, \dots, \gamma)) =
\Dc_{\kappa, r}(f, g) \mid_{\tr(\kappa + (r+1)\lambda) + 2r} \gamma
\end{equation}
for $\kappa, \lambda \in \Zb^d$ and $\gamma \in \SL_2(\Rb)$.
The upshot of this operator is the following result.

\begin{lemma}
  \label{lemma:lower}
For $\kappa = (k_1, \dots, k_d) \in \Zb$, let $f : \Hb^d \to \Cb$ be a real-analytic function that is harmonic in $\tau_1$ of weight $k_1$ and holomorphic in $\tau_2, \dots, \tau_d$.
  For any holomorphic function $g: \Hb^d \to \Cb$, we have
  \begin{equation}
    \label{eq:lower}
    L_{\tau} \Dc_{\kappa, r}(f, g) =
 \binom{k_1 + r - 1}{r} 
    (g^\Delta)^{r+1} \tR^r_{\tr(\kappa) - 2} ((L_{\tau_1} f )/g)^\Delta
  \end{equation}
  for all $r \in \Nb_0$.
\end{lemma}

\begin{proof}
  This follows directly from the definition and equation \eqref{eq:LRC}.
\end{proof}
We can also componentwisely apply $\Dc_{\kappa, r}$ when $f$ is vector-valued, in which case we also write $\Dc_{\kappa, r}(f, g)$, and the result above holds as well.

\subsection{Quadratic Space and Shimura Variety}
\label{subsec:QS}
Let $V$ be a finite dimensional $F$-vector space of dimension $n+2 \ge 0$ with a non-degenerate quadratic form $Q$.
For our purpose, $n$ is even when $d \ge 2$, i.e.\ $F \neq \Qb$.
Denote $V_{\sigma_j} := V \otimes_{F, \sigma_j} \Rb$ for $1 \le j \le d$, which is an $\Rb$-quadratic space of signature $(p_j, q_j)$, and $V(\Rb) = V \otimes_\Qb \Rb = \oplus_i V_{\sigma_i}$ is an $\Rb$-quadratic space of signature $(p, q)$ with $p = \sum_j p_j, q = \sum_j q_j$. 
We say that $V$ is totally positive if $V(\Rb)$ is positive definite. %
The symmetric domain $\Db$ associated to $\SO(V(\Rb))$ is realized as the Grassmannian of $q$-dimensional negative definite oriented subspaces of $V(\Rb)$. It consists of 2 components unless $q = 0$, in which case it is a point.

Let  $\GSpin_V$ be the general spin group of $V$.
We will be interested in the $\Qb$-algebraic group
\begin{equation}
  \label{eq:H}
H = H_V := \Res_{F/\Qb} \GSpin_V,
\end{equation}
which fits into the exact sequence
 \begin{equation}
   \label{eq:Hext}
1 \to Z \to H \to \Res_{F/\Qb} \SO(V) \to 1
 \end{equation}
with $Z(\Qb) \cong  F^\times$.
Denote $\nu: C(V) \to  F^\times$ the spinor norm on the Clifford algebra $C(V)$ of $V$, which induces a surjection $\nu: H \to T := \Res_{F/\Qb} \mathbb{G}_m$ of algebraic groups. 

\begin{exmp}
  \label{exmp:H2}
  More generally, the group $\GSpin$ can be defined for a quadratic module $M$ over a commutative ring $R$. 
  For a nice example, we consider the hyperbolic plane, where $M = R^2$ is a free $R$-module with quadratic form $Q((a, b)) = ab$.   Furthermore denote $e_1, e_2$ the images of $(1, 0), (0, 1) \in M$ in the Clifford algebra $C(M)$, and $e_0 := e_1 e_2, e_3 := e_2 e_1 \in C(M)$. 
  Then we have $e_0 + e_3 = 1$ in $C(M)$ and an $R$-algebra isomorphism
  $$
C(M) = \bigoplus_{i = 0}^3 R \cdot e_i \cong M_2(R),~ (a_0, a_1, a_2, a_3) \mapsto \pmat{a_0}{a_1}{a_2}{a_3}.
$$
The even Clifford algebra $C^0(M)$ corresponds precisely to the diagonal matrices in $M_2(R)$. The group $\GSpin$ then consists of invertible diagonal matrices, and the spinor norm $\nu$ is just the determinant. 
\end{exmp}

For the rest of this subsection, suppose $V$ has signature
\begin{equation}
  \label{eq:sign}
 ((n, 2), (n+2, 0), \dots, (n+2, 0)),~ 2 \mid n.
\end{equation}
Then the hermitian symmetric space associated to $H$ can be realized as the Grassmannian $\Db = \Db_V = \Db^+ \sqcup \Db^-$ of oriented negative-definite 2-planes of $V_{\sigma_1}$.
If we denote $V_\Cb := V \otimes_{F, \sigma_1} \Cb$ and extend the quadratic form $\Cb$-bilinearly to $V_\Cb$, then we can identify $\Db$ with the quadric
\begin{equation}
  \label{eq:Hc}
  \Hc := \{[Z]\in \Pb(V_\Cb): (Z, Z) = 0,~ (Z, \bar Z) < 0\},
\end{equation}
in the projective space $\Pb(V_\Cb)$ by sending $[Z = X +iY]$ to the oriented 2-plane spanned by the ordered basis $\{X, Y\} \subset V_{\sigma_1}$.
This endows $\Db$ with a complex structure.
We can furthermore identify the tube domain
$$
\Hs := \{z \in V_0 \otimes_\Rb \Cb: Q(\Im(z)) < 0\},
$$
where $V_0 := V_{\sigma_1} \cap a^\perp \cap b^\perp$ for isotropic vectors $a, b \in V_{\sigma_1}$ with $(a, b) = 1$, with $\Hc$ by sending $z$ to the class of
$$
w(z) := z + a - Q(z) b
$$
in $\Pb(V_\Cb)$.
For $\gamma \in H(\Rb)$, we have the automorphy factor $j(\gamma, z) = (\gamma w(z) , b)$ from
$$
\gamma w(z) = j(\gamma, z) w(\gamma z).
$$

For $z \in \Db^\pm$, denote $\bar z \in \Db^\mp$ the 2-plane with the opposite orientation.
The subgroup of $H(\Rb)$ fixing $\Db^+$ is the subgroup $H(\Rb)_+$ consisting of elements with totally positive spinor norm.
For a compact open $K \subset H(\hat\Qb)$, the $\Cb$-points of the Shimura variety associated to $H$
\begin{equation}
  \label{eq:XK}
X_{V, K} =  X_K := H(\Qb) \backslash \Db \times H(\hat\Qb) /K
\end{equation}
is a complex quasi-projective variety of dimension $n$, and has a canonical model over $\sigma_1(F)$ \cite{Shih}.
When $V$ is anisotropic over $F$, the variety $X_K$ is projective.

\begin{exmp}
  \label{exmp:M2Q}
  For $F = \Qb$ and $(V, Q) = (M_2(\Qb), \det)$, we have
  $$
  \Hb^2 \cup (\Hb^-)^2 \cong \Db,~ (z_1, z_2) \mapsto  \Rb \Re Z + \Rb \Im Z,
  $$
  where the line spanned by
  $Z = Z(z_1, z_2) := \smat{z_1}{-z_1z_2}{1}{-z_2} \in V(\Cb)$ is in the quadric $\Hc$ defined in \eqref{eq:Hc}.
  For a congruence subgroup $\Gamma \subset \SL_2(\Zb)$, there exists compact open $K_\Gamma \subset H(\hat\Qb)$ such that the connected component of the Shimura variety $X_{V, K_\Gamma}$
  can be identified with the product of modular curves $X_\Gamma \times X_\Gamma$. See section 3.1 in \cite{YY19} for more details.
\end{exmp}

A  meromorphic modular form on $X_K$ of weight $w \in \Zb$ is a collection of meromorphic functions $\Psi(\cdot, h): \Hs \to \Cb$ for each $h \in H(\hat \Qb)$ satisfying
\begin{equation}
  \label{eq:Omod}
  \begin{split}
    \Psi(z, hk ) &= \Psi(z, h) \text{ for all } k \in K,\\
    \Psi(\gamma z, \gamma h) &= j(\gamma, z)^w \Psi(z, h),~ \text{ for all }\gamma \in H(\Qb)
  \end{split}
\end{equation}
and are meromorphic at the boundary\footnote{The boundary behavior is relevant for us when $d = 1$ as $V$ will otherwise be anisotropic.}.
For such a meromorphic modular form, we also denote
\begin{equation}
  \label{eq:Pet1}
 \|\Psi(z, h)\|_{\mathrm{Pet}} := |\Psi(z, h)| \cdot |y|^w,
\end{equation}
which is a real-analytic function
on $X_K$ (see section 2 of \cite{Br12}). 

To describe the connected components of $X_K$, we write
$$
H(\hat\Qb) = \coprod_j H(\Qb)_+ h_j K,
$$
where $H(\Qb)_+ = H(\Qb) \cap H(\Rb)_+$. Then for $n > 0$, we have $X_K \cong \coprod_{j} \Gamma_{h_j} \backslash \Db^+$ and (1.8) of \cite{Kudla97} gives us
\begin{equation}
  \label{eq:pi0}
  \pi_0(X_K) \cong H(\Qb)_+ \backslash H(\hat\Qb) / K \stackrel{\nu}{\cong}
  F^\times_+ \backslash   \hat F^\times/  \nu(K) \cong \Gal(E_K /F),  
\end{equation}
where $E_K/F$ is a finite Galois extension that the connected component $Y_K := \Gamma_1 \backslash \Db^+$ is defined over.
Furthermore, for $\sigma \in \Gal(E_K/F)$ associated to $\nu(h_j^{-1})$, we have $Y_K^\sigma \cong Y_{h_j K h_j^{-1}}$ over $E_K$ and
$$
X_K \cong \coprod_j Y_{h_j K h_j^{-1}} \cong \coprod_{\sigma \in \Gal(E_K/F)} Y_K^\sigma.
$$
When restricted to the center $Z$ in \eqref{eq:Hext}, the map $\nu$ above is simply the square map and its image consists of square elements in $\Gal(E_K/F)$.

When $n = 0$, the domain $\Db$ has two points and the group $\GSpin_V$ can be identified with $E_W^\times$ for a totally imaginary, quadratic extension $E_W$ over $F$, the norm from $E_W$ to $F$ is simply the spinor norm, and
\begin{equation}
  \label{eq:pi0a}
  X_K \cong \Db \times \Ab_{E_W}^\times / E_W^\times (E_W)^\times_\infty K.
\end{equation}

For $1 \le j \le d$, there is a unique $F$-quadratic space $V(j)$ with signature
$$
\mathrm{sig}(V(j)) = ((n+2, 0), \dots, (n+2, 0), (n, 2), (n+2, 0), \dots, (n+2, 0))
$$
and isomorphic to $V$ at all finite places.
They are neighboring quadratic spaces of an admissible incoherent quadratic space $(\Vb, Q)$ over $\hat F$ (see \cite[section 7]{BY11}).
One can carry out the construction before \eqref{eq:XK} to define $X_{V(j), K}$,
which is the $\Cb$-points of a Shimura variety defined over $\sigma_j(F)$.
There is a quasi-projective variety $\Xb_K$ defined over $F$ such that the base change to $\sigma_j(F)$ is $X_{V(j), K}$, and the union of $X_{V(j), K}$ over all $j$ is the $\Cb$-points of $\Xb_K$ considered as a scheme over $\Qb$ (see Lemma 7.1 of \cite{BY11}).

\subsection{Unimodular Lattice}
\label{subsec:unilat}
An $\Oc$-lattice  $L \subset V$, i.e.\ a finitely generated $\Oc$-module satisfying $L \otimes_\Oc F = V$, is called even, resp.\ integral, if $Q(L)$ is in $\df^{-1}$, resp.\ $\Oc$. 
For an even $\Oc$-lattice $(L, Q)$, the quadratic form $Q_\Zb(x):= \tr_{F/\Qb}Q(x)$ is $\Zb$-valued, and we denote
\begin{equation}
  \label{eq:L'}
  L' := \{y \in V: (y, L) \subset \df^{-1} \},
\end{equation}
which is the dual of the $\Zb$-lattice $L$ with respect to $Q_\Zb$.
For $\mu \in L'/L$ and $m \in F$, we write
\begin{equation}
  \label{eq:Lmmu}
  L_{m, \mu} := \{\lambda \in L + \mu: Q(\lambda) = m\} \subset L'.
\end{equation}
which is empty if $m \not\in  \df^{-1} + Q(\mu)$. 

Also, we denote
\begin{equation}
  \label{eq:SL}
 S_L := \oplus_{\mu \in \hat L'/\hat L} \Cb \phi_\mu \subset \Sc(V(\hat F)),~ \phi_\mu := \mathrm{char}(\hat L + \mu)
\end{equation}
the subspace of Schwartz functions with support on $\hat L' := L' \otimes \hat \Oc $ and constant on $\hat L := L \otimes \hat \Oc$.
Note that $\hat L' / \hat L = L' /L$ is a finite abelian group, and we write $\hat L_{m, \mu} := L_{m, \mu} \otimes \hat \Oc$.
For any $h \in H(\hat \Qb)$, the lattice \footnote{The inverse in this definition makes the action of $h$ a right action.}
  \begin{equation}
    \label{eq:L1}
   L_h := V \cap h^{-1} \cdot \hat L \subset V
 \end{equation}
  satisfies $h^{-1} \cdot \hat L = \hat L_h$ and is in the same genus as $L$. 
Using the left action of $h^{-1}$, we identify
\begin{equation}
  \label{eq:quadmodid}
L'/L = \hat L'/\hat L \cong 
 \hat L'_h/\hat L_h = L'_h/L_h.
\end{equation}
The linear isomorphism $\iota_h: S_L \to S_{L_h}$, which sends $\phi_\mu$ to $\phi_{h^{-1} \mu}$, then identifies $\rho_L$ with $\rho_{L_h}$. 

We say that a lattice $L$ is \textit{$\Zb$-unimodular} if $L' = L$.
Then the set
\begin{equation}
  \label{eq:UF}
  \Uc_F := \{(L, Q): L \text{ is an even } \Zb\text{-unimodular } \Oc\text{-lattice}\}
\end{equation}
is a commutative monoid with respect to $\oplus$. 
Let $\Uc_F^+ \subset \Uc_F$ denote the semigroup consisting of totally positive, non-trivial lattices.
This set is non-empty by the following result.
\begin{prop}
  \label{prop:unimodular}
For any  totally real field $F$, the semigroup $\Uc_F^+$ is non-trivial.
\end{prop}

\begin{proof}
  For an integral $\Oc$-lattice $L \subset V$, let $\df(L) \subset \Oc$ be the discriminant ideal of $L$ (see \cite{Chang70}). 
Then $\df(L) = \Oc$  if and only if $L^\# = L$, where $L^\# := \{y \in V: (y, L) \subset \Oc\} $ is the $\Oc$-dual of $L$. 
In this case, $L$ is called unimodular.
Satz 1 in \cite{Chang70} gives a necessary and sufficient condition for the existence of definite unimodular $\Oc$-lattices, which is easily seen to be satisfied when $V = W^{\oplus 4}$ with $ \dim_F(W)$ is divisible by 2.
Furthermore, for any one of the $2^d$ possible signatures for definite spaces, there is a space $V$ having this signature and containing a unimodular $\Oc$-lattice.
So for any $\alpha \in F^\times$, there is a definite space $V$ such that it becomes totally positive definite after scaling its quadratic form by $\alpha$. 

It is a well-known result of Hecke (see the last Theorem in \cite{Weil74}) that the class of $\df$ in $\Cl(F)$ is a square. 
So we can write $\df^{-1} = \af^2 (\delta)$ with $\af \subset \Oc$ and $\delta \in F$. 
Let $(L, Q)$ be a non-trivial, integral unimodular $\Oc$-lattice such that $\delta Q$ is totally positive definite. 
Then $(\af L, \delta Q)$ is an even $\Oc$-lattice and
$$
\delta (\lambda, \af L)  \subset \df^{-1} = \af^2 (\delta) 
\Leftrightarrow 
(\lambda, \af^{-1} L) \subset \Oc
$$
for all $\lambda \in (\af L)'$. 
So $\lambda \in \af L^\# = \af L$ and $(\af L, \delta Q)$ is non-trivial, $\Zb$-unimodular and totally positive definite. 
\end{proof}

\subsection{Special Cycles.}
\label{subsec:cycle}
Now suppose $V$ decomposes as $W \oplus U$ such that $U$ is totally positive subspace of dimension $r$.
Then the Grassmannian $\Db_U$ of $U$ consists of one point $z_U$, and
$\Db_W$ can be realized as  an analytic submanifold of $\Db$ via
$$
z_W \mapsto (z_W, z_U) \in \Db_V.
$$
Similarly, the algebraic group $H_W :=  \Res_{F/\Qb}\GSpin_W$, resp.\  $H_U :=  \Res_{F/\Qb}\GSpin_U$, is isomorphic to the  pointwise stabilizer of $U$, resp.\ $W$, in $H_V$, which induces $H_W \times H_U \hookrightarrow H_V$ and we write
\footnote{We will sometimes view $H_U, H_W$ as subgroups of $H_V$ to lighten the notation.}
$(h_W, h_U) \in H_V$ for $h_W \in H_W, h_U \in H_U$.
Then for $h \in H_V(\hat\Qb)$, the image of the natural map
$$
H_W(\Qb) \backslash \Db_W \times H_W(\hat\Qb) /H_W(\hat\Qb) \cap h K h^{-1},~ (z, h_1) \mapsto (z, h_1 h)
$$
defines a codimension-$r$ cycle on $X_K$, denoted by $Z(W, h)$. 
A word of caution about the notation: in \cite{Kudla97}, the items $\Db_W, H_W$ and $Z(W, h)$ were defined with $W$ replaced by $U$. We decide to change the notation here as $U$ will be varying later and it is important to keep track of $W$.

When $r = n$, the set $\Db_W = \{z_W^\pm\}$ consists of two elements and points of $Z(W, h)$ are called (small) CM points.
For a subfield $F_0 \subset F$, we can consider the $F_0$-quadratic space
\begin{equation}
  \label{eq:restrict}
W_{F_0} := \Res_{F/F_0} W.
\end{equation}
For any $F_0$-quadratic space $V_0 = W_{F_0} \oplus U_0$ with $U_0$ totally positive,
the image of the above homomorphism $H_{W_{F_0}} \hookrightarrow H_{V_0}$ is a torus denoted by $T:= T_W \subset H_{V_0}$. 
For any open compact $K \subset H_{V_0}(\hat \Qb)$, the torus $T$ gives rise to the CM 0-cycle $Z(W_{F_0}, h)$ on $X_{V_0, K}$ defined over $F$. Its $\Cb$-points are given by
\begin{equation}
  \label{eq:ZW}
  Z(W_{F_0}, h)(\Cb) =  T(\Qb) \backslash \{z_W^\pm\}  \times T(\hat \Qb) / K^h_T \to X_{V_0, K},~
  [z_W^\pm, t] \mapsto [z_W^\pm, th],
\end{equation}
 where $K^h_T := h K h^{-1} \cap T(\hat \Qb)$.
 These were called ``big CM points'' in \cite{BKY12} when $F_0 = \Qb$ and $U_0$ is trivial.
We omit $h$ from the notation when it is trivial.

To obtain a 0-cycle defined over $\Qb$, one considers the 0-cycle
\begin{equation}
  \label{eq:ZWb}
  Z(\Wb) := \sum_{1 \le j \le d,~ \tau_j = \sigma_j \circ \sigma_1^{-1}:\Rb \to \Rb}
  \tau_j (Z(W_\Qb)),
\end{equation}
where $\Wb$ is the admissible incoherent quadratic space with neighbors $W(j)$.
Note that $Z(\Wb)$ is $Z(W)$ in Equation (2.13) of \cite{BKY12}. 
The 0-cycles $\tau_j(Z(W))$ can be constructed as above with $W$ replaced by $W(j)$ for $1 \le j \le d$ (see \cite[Lemma 2.2]{BKY12}).

\begin{exmp}
  \label{exmp:biquad}
  We follow the discussions in \cite[section 3]{YY19} and \cite[section 3.2]{Li21} to realize the CM points appearing in Theorem \ref{thm:GZ} as big CM points.
  Let $d_1, d_2 < 0$ be discriminants such that $F  = \Qb(\sqrt{D})$ is a real quadratic field, where $D := d_1 d_2> 0$.
  We label the real embeddings $\sigma_j: F \to \Rb$ such that $\sigma_1(\sqrt{D}) = \sqrt{D}$ and $\sigma_2(\sqrt{D}) = - \sqrt{D}$.
  Then $E := E_1E_2$ is a CM extension of $F$ and becomes an $F$ quadratic space $W$ of signature $((0, 2), (2, 0))$ with respect to the quadratic form $Q(\mu) := -\frac{ N \mu\bar\mu}{\sqrt{D}}$.
For $i = 1, 2$, let $z_i = \frac{-b_i + \sqrt{d_i}}{2a_i}$ be a  CM point of discriminant $d_i$. Denote $H_i/E_i$ the  ring class field corresponding to $z_i$,  $H = H_1H_2$ and
\begin{equation}
  \label{eq:Gi}
G_i :=  \Gal(H_i/E_i) \subset \tG_i := \Gal(H_i/\Qb).  
\end{equation}
The group $\Gal(H/\Qb)$ embeds into $\tilde{G_1} \times \tilde{G_2}$ via restriction, under which the image of $\Gal(H/E)$ is a subgroup of $G_1 \times G_2$. 
If $(V, Q) = (M_2(\Qb), \det)$, then
  the map\footnote{  Here $(,)$ is the bilinear pairing on $M_2(\Cb)$ induced by the determinant.}
  $$
  V \to W_\Qb,~ \gamma \mapsto (\gamma, Z( z_1, z_2))
  $$
  is an isometry with $N = a_1a_2$.
  The CM 0-cycle $Z(W_\Qb)$ defined in \eqref{eq:ZW} is given by
  $$
  Z(W) = \sum_{\sigma \in \Gal(H/E)} (z_1, z_2)^\sigma + (-\overline{z_1}, -\overline{z_2})^\sigma
  = \sum_{ \substack{(\sigma_1, \sigma_2) \in G_1 \times G_2\\ \sigma_1\mid_{H_0} = \sigma_2\mid_{H_0}}}
  (z_1^{\sigma_1}, z_2^{\sigma_2}) + (-\overline{z_1}^{\sigma_1}, -\overline{z_2}^{\sigma_2}),
$$
where $H_0 = H_1 \cap H_2$.
Note that $H_0 = \Qb$ when $d_1, d_2$ are co-prime. 
On the other hand, we have
$$
  Z(W(2)_\Qb) = \sum_{\sigma \in \Gal(H/E)} (z_1, -\overline{z_2})^\sigma + (-\overline{z_1}, {z_2})^\sigma.
  $$
  Lemma 3.2 in \cite{Li21} tells us that $H_0/\Qb$ is abelian.
  Its proof even implies that every element in  $\Gal(H_0/\Qb)$ has order dividing 2.
  From these, we then know that the element $\sigma_2 \in \Gal(H_2/E_2)$ satisfying $z_2^{\sigma_2} = -\overline{z_2}$ is a square and hence trivial when restricted to $H_0$.
Therefore, $Z(W_\Qb) = Z(W(2)_\Qb)$, and $Z(W_\Qb)$ is already defined over $\Qb$.  %
\end{exmp}

On the other extreme, when $r = 1$, we have $W = (F x_0)^\perp$ for some $x_0 \in F$ with $Q(x_0) = m \gg 0$, and the cycle $Z(W, h)$ is a divisor.
We define a weighted divisor by the finite sum
\begin{equation}
  \label{eq:Zmphi}
  Z(m, \phi) := \sum_{H_{W}(\hat\Qb) \backslash H(\hat\Qb) / K} \phi(h^{-1} x_0) Z(W, h)
\end{equation}
for any $\phi \in \Sc(V(\hat F))^K$.
We also write
\begin{equation}
  \label{eq:Zmmu}
  Z(m, \mu) := Z(m, \phi_\mu)
\end{equation}
for $\mu \in L'/L$ and $L \subset V$ an even $\Oc_F$-lattice.

\subsection{A helpful Lemma.}
In this section, we record a result that will be helpful in studying zeros of definite theta functions.

\begin{lemma}
  \label{lemma:kronecker}
  Let $\theta_1, \dots, \theta_N \in \Rb$ be $\Qb$-linearly independent irrational numbers, $M \in \Nb$ and $b$ an integer with $0 \le b \le M-1$.
  For any $\alpha_1, \dots, \alpha_N \in \Rb$, there exists an infinite subsequence $\{n_i: i \in \Nb\} \subset \Nb$  such that
  $$
\lim_{i \to \infty} \ebf((Mn_i + b)\theta_j) = \ebf(\alpha_j)
$$
for all $1 \le j \le N$. 
\end{lemma}

\begin{proof}
  By replacing $\theta_j$ with $M\theta_j$ and $\alpha_j$ with $\alpha_j - b\theta_j$, we can suppose that $M = 1$ and $b = 0$.
  We first prove the case $\alpha_j = 0$ for all $1\le j \le N$ by constructing the sequence $\{n_i\}$ inductively.
  For any $\epsilon > 0$, 
  by Kronecker's approximation theorem \cite[Theorem 7.10]{Apostol90}, there exists $n, h_j \in \Zb$ such that
  $$
|n \theta_j - h_j - \epsilon/3| < \epsilon / 3
$$
for all $1 \le j \le N$, which is equivalent to
$$
0  < n\theta_j - h_j < 2\epsilon/3. 
$$
By replacing $n, h_j$ with $-n, -h_j$, we can ensure that $n \ge 1$ for $\epsilon < 3/2$, while $|n\theta_j - h_j| < \epsilon$ still holds.
Denote $n(\epsilon) :=n$ and $n_i := \max_{} \{n(1/i'): 1 \le i' \le i\}$. 
Then
$$
\lim_{i \to \infty} \ebf(n_i\theta_j) = 1,~ 1 \le j \le N,
$$
and the sequence $\{n_i\}$ is infinite since $\theta_j$'s are irrational.

In the general case, we can first use Kronecker's approximation theorem to produce a sequence $\{n'_i \} \subset \Zb$ such that $\lim_{i \to \infty} \ebf(n'_i \theta_j) = \ebf(\alpha_j)$ for all $1 \le j \le N$. 
For each $i$, we can find $i' > i$ such that $n''_i := n'_i + n_{i'}$ forms an increasing sequence in $\Nb$, where $\{n_i\}$ is the sequence we have constructed in the case all $\alpha_j$'s are 0. Then the new sequence $\{n''_i\}$ satisfies the condition of the lemma.
\end{proof}

\begin{lemma}
\label{lemma:seriesvanish}
  Suppose $\alpha_i, c_i \in \Cb$ for $i \in \Nb$ satisfy the condition that the series defining
$$
\phi(s) := \sum_{i \in \Nb} \alpha_i c_i^s
$$
converges absolutely for $s = s_0 \in \Rb$ and equals to 0 for all but finitely many $s \in \Zb_{> s_0}$. Then we have
\begin{equation}
  \label{eq:coeff0}
\sum_{i \in \Nb,~ c_i = c} \alpha_i = 0  
\end{equation}
for any $c \in \Cb^\times$.
In other words,  $\phi(s)$ is identically zero.
\end{lemma}

\begin{rmk}
  \label{rmk:Freitag}
  If $\alpha_i > 0$ for all $i \in \Nb$, then $c_i = 0$ for all $i \in \Nb$. This strengths Lemma $5.6_1$ in \cite{Freitag}.
\end{rmk}
\begin{proof}
  Without loss of generality, we take $s_0 = 0$. 
After rearranging and scaling all the $c_i$'s, we can suppose that $|c_i| \ge |c_{i+1}|$ for all $i \in \Nb$, and $1 := |c_1| = |c_m| > |c_{m+1}|$. 
Denote 
$$
\phi_1(s) := \sum_{i = 1}^m \alpha_i c_i^s.
$$
Using induction, it is then enough to prove the lemma for $\phi_1(s)$. 

The condition $\phi(s) = 0$ for all but finitely many $s \in \Nb$ implies that
$$
\lim_{s \to \infty} \left| \phi_1(s) \right| = 
\lim_{s \to \infty} \left| \phi(s) - \phi_1(s) \right| = 0.
$$
For $1 \le i \le m$ with $c_i \neq 0$, we can now write 
$$
c_i =  \zeta^{r_{i, 0}}\ebf \lp \sum_{j = 1}^N r_{i, j} \theta_j \rp
$$ 
with $\zeta = \ebf(1/M)$ for some $M \in \Nb$, $(r_{i, j})_{0 \le j \le N} \in \Zb/M\Zb \times \Zb^N$ and $\theta_j \in \Rb$ such that $1, \theta_1, \dots, \theta_N$ are  $\Qb$-linearly independent.
Then we have $c_i = c_{i'}$ if and only if $r_{i, j} = r_{i', j}$ for all $0 \le j \le N$. 
For any integer $0 \le b \le M-1$,  we have
$$
\phi_1(s) = f_b(\ebf(s \theta_1), \dots, \ebf(s \theta_N)),~ s \in M \Nb + b,
$$
where $f_b(z_1, \dots, z_n) : = \sum_{i = 1}^m \alpha_i (\zeta^b)^{r_{i, 0}} \prod_{j = 1}^N z_j^{r_{i, j}} \in \Cb[z_1, z_1^{-1}, \dots, z_N, z_N^{-1}]$.

For any $\beta \in (\Rb/\Zb)^N$, Lemma \ref{lemma:kronecker} implies that 
there exists an infinite subsequence $\{n_k: k \in \Nb\} \subset \Nb$ such that 
$$
\lim_{k \to \infty} (\ebf((Mn_k + b) \theta_1), \dots, \ebf((M n_k + b) \theta_N)) = \beta. 
$$
Since $f_b$ is continuous, we then have
$$
f_b(\beta) = \lim_{k \to \infty} f_b (\ebf((M n_k + b) \theta_1), \dots, \ebf((M n_k + b) \theta_1)) = \lim_{s \to \infty} \phi_1(s) = 0. 
$$
This implies that the polynomial $f_b$ is identically zero, or equivalently
$$
\sum_{1 \le i \le m,~ (r_{i, j})_{1 \le j \le N} = r}
\alpha_i (\zeta^{r_{i, 0}})^b
 = 0
$$
for all $r \in \Zb^N$ and $b \in \Zb/M\Zb$.
This then implies
$$
\sum_{1 \le i \le m,~ (r_{i, j})_{0 \le j \le N} = r'}
\alpha_i
 = 0
 $$
 for all $r' \in \Zb/M\Zb \times \Zb^N$  and $\phi_1(s)$ is identically 0.
So the lemma holds for $\phi_1$, and hence also for $\phi$ by induction.
\end{proof}
As an immediate consequence, we have the following result.
\begin{cor}
\label{cor:bijection}
Let $\{c_i: i \in \Nb\}, \{b_j: j \in \Nb\}$ be subsets of $\Cb$, and
$$
I := \{i \in \Nb: c_i \neq 0\},~
J := \{j \in \Nb: b_j \neq 0\}. 
$$
Suppose there is $s_0 \in \Rb$ such that the series
$$
f(s) := \sum_{i \in \Nb} c_i^s,~ 
g(s) := \sum_{j \in \Nb} b_j^s
$$
converge absolutely for $s = s_0$ and $f(s) = g(s)$ for all but finitely many $s \in \Zb_{> s_0}$, then there is a bijection $\sigma: I \to J$ such that $c_i = b_{\sigma(i)}$ for all $i \in I$.
In particular, $I$ is finite if and only if $J$ is finite, in which case they have the same size. 
\end{cor}

\section{Functions}
\label{sec:func}
\subsection{Weil Representation and Theta Function}
Let $G := \Res_{F/\Qb}(\SL_2)$ and
$$
\Gamma := \SL_2(\Oc_F) \subset G(\Rb) \subset G(\Ab).$$
Also denote
$\Gamma_f := \SL_2(\hat\Oc) \subset G(\hat\Qb)$.

For an $F$-quadratic space $(V, Q)$ of even dimension, let $\omega = \omega_\psi$ be the Weil representation of $ G(\Ab)$ on the space of Schwartz functions $\Sc(V(\Ab_F)) = \Sc(V(\hat F)) \otimes  \bigotimes_{1 \le j \le d} \Sc(V_{\sigma_j})$, where $\psi$ is the standard additive character on $F \backslash \Ab_F$.

At the infinite local place, suppose $(W, Q_W)$ is an $\Rb$-quadratic space of signature $(p, q)$ with $2 \mid (p + q)$.
For a point $w$ in the symmetric space $\Db_W$ associated to $\SO(W)$, we obtain an orthogonal decomposition $W = w^\perp \oplus w$ and a Schwartz function
\begin{equation}
  \label{eq:varphipq}
  \phi^{p, q}(w, \lambda) := \exp(-2\pi (Q_W(\lambda_{w^\perp})  + Q_W(\lambda_w)))
\end{equation}
in $\Sc(W)$, which is acted on by
$\SL_2(\Rb)$
via the Weil representation $\omega_W$ to produce
\begin{equation}
  \label{eq:varphiWtau}
  \phi^W(\tau, w, \lambda) := v^{q/2} \ebf(Q(\lambda_{w^\perp}) \tau + Q(\lambda_w) \overline{\tau}).
\end{equation}
Note that
\begin{equation}
  \label{eq:wbar}
  \phi^W(\tau, w, \lambda) = \phi^W(\tau, \bar w, \lambda)  
\end{equation}
as %
it is independent of the orientation of $w$. 
If $q = 0$, then the expressions in \eqref{eq:varphipq} and \eqref{eq:varphiWtau} are independent of $w$ and we omit them from the notation.
In addition, we will also be interested in the following ``singular Schwartz function'' when $q \ge 1$
\begin{equation}
  \label{eq:phi*}
  \phi^{W, *}(\tau, w, \lambda) := 
\ebf(Q(\lambda) \tau)
  \begin{cases}
- (-4\pi Q(\lambda_w))^{1 - q/2} \Gamma(q/2 - 1, -4\pi Q(\lambda_w) v)    ,& \lambda \not\in w^\perp,\\
\mathrm{CT}_{s = 0 } \frac{v^{s + q/2 - 1}}{s + q/2 - 1},& \lambda \in w^\perp.
  \end{cases}
\end{equation}
Here $\Gamma(s, x) := \int^\infty_x t^{s-1} e^{-t} dt$ is the incomplete Gamma function.
Direct calculations yield
\begin{equation}
  \label{eq:Lphi*}
  L_\tau \phi^{W, *}(\tau, w, \lambda) = \phi^W(\tau, w, \lambda).
\end{equation}
for all $w \in \Db_W$ and $\lambda \in W$.

To describe the finite local place, let $L \subset V$ be an even lattice and $S_L \subset \Sc(V(\hat F))$ the subspace as in \eqref{eq:SL}.
Then $\Gamma$ acts on the space $S_L$ via $\omega$, whose complex conjugate we denote by $\rho_L$.
Its explicit values on the generators of $\Gamma$ can be found in section 3.2 of \cite{Br12}.
Furthermore, it is unitary with respect to the hermitian pairing on $S_L$
\begin{equation}
  \label{eq:pair0}
  \langle\phi, \psi \rangle := \sum_{\mu \in L'/L} \phi(\mu) \overline{\psi(\mu)}
  ,~ \phi, \psi \in S_L.
\end{equation}
More generally, for lattices $L, M$ and $\phi \in S_{M \oplus L}, \psi \in S_{M}$, we define
$  \langle \phi, \psi \rangle_L \in S_L$ by
\begin{equation}
  \label{eq:pair1}
  \langle \phi, \psi \rangle_L (\mu) := \sum_{\nu \in M'/M} \phi((\nu, \mu)) \overline{\psi(\nu)},~ \mu \in L'/L.
\end{equation}
These pairings are then naturally defined for functions valued in $S_L$.

For the rest of the section, suppose $V_{\sigma_j}$ is positive definite whenever $j \ge 2$. 
Then $\Db_{V_{\sigma_j}}$ is $\Db$ for $j = 1$ and a point otherwise. 
For $\vtau = (\tau_j) \in \Hb^d, z \in \Db, h \in H(\hat\Qb)$ and $\phi_f \in \Sc(V(\hat F))$, we define the Siegel theta function
\begin{equation}
  \label{eq:thetaphif}
  \begin{split}
      \Theta(\vtau, z, h; \phi_f) &:= \sum_{\lambda \in V(F)} \phi_f(h^{-1} \lambda)  \phi_\infty(\tau, z,\lambda ),\\
  \phi_\infty(\tau, z, \lambda) &:=
 \phi^{V_{\sigma_{1}}}(\tau_1, z,\lambda )
  \bigotimes_{2 \le j \le d} \phi^{V_{\sigma_j}}(\tau_j,\lambda ).
  \end{split}
\end{equation}
Note that $\Theta(\tau, \bar z, h; \phi_f) = \Theta(\tau, z, h; \phi_f)$ by \eqref{eq:wbar}.
In the variable $\vtau$, this is a Hilbert modular form of weight
\begin{equation}
  \label{eq:kaV}
 \tka = \tka(V) := ((p_1 - q_1)/2, n/2+1 \dots, n/2 + 1) \in \Zb^d. 
\end{equation}
For any even lattice $L \subset V$, we also denote the associated theta function by
\begin{equation}
  \label{eq:ThetaL}
  \Theta_L(\vtau, z, h) := \sum_{\mu \in L'/L} \Theta(\vtau, z, h; \phi_\mu) \phi_\mu,
\end{equation}
which is valued in $S_L$.
When $h = 1$, we omit it from the notation.
The definition of $\Theta_L$ implies that
\begin{equation}
  \label{eq:HactTheta}
  \Theta_{L_{}}(\tau, z) =
\Theta_{L_{h^{}}}(\tau, z, h),
\end{equation}
where $L_{h^{}}$ is defined in \eqref{eq:L1} and $L'/L \cong L_h'/L_h$ via \eqref{eq:quadmodid}.

If $V_{\sigma_1}$ is definite, then $\Theta_L(\tau, z, h)$ is independent of $z$ and we write
\begin{equation}
  \label{eq:thetaL}
 \theta_L(\vtau, h) := \Theta_L(\vtau, z, h).
\end{equation}
In particular when $L = P$ is positive definite and $\Zb$-unimodular, this is a scalar-valued, holomorphic Hilbert modular form on $\Gamma_F$ of parallel weight $n/2 + 1$. 
We denote the graded subring
\begin{equation}
  \label{eq:MTheta}
  \Mc^\theta_{F} := \mathrm{Span}\{\theta_P(\tau): P \in \Uc_F^+
\} \subset \Mc_{F}.
\end{equation}

For future convenience, we also define 
\begin{equation}
  \label{eq:Theta*}
  \begin{split}
      \Theta^*(\vtau, z, h; \phi_f) &:= \sum_{\lambda \in V(F)} \phi_f(h^{-1} \lambda)  
 \phi^{V_{\sigma_{1}}, *}(\tau_1, z,\lambda )
  \otimes_{2 \le j \le d} \phi^{V_{\sigma_j}}(\tau_j,\lambda ),\\
\Theta^*_L(\tau, z, h) &:= \sum_{\mu \in L'/L} \Theta^*(\vtau, z, h; \phi_\mu) \phi_\mu.
  \end{split}
\end{equation}
The sums converge absolutely since the singular Schwartz function decays as a Schwartz function and $z^\perp \cap L'$ is contained in a positive definite lattice. 
For fixed $(z, h)$, it defines a real-analytic function in $\tau$, which satisfies the analogue of \eqref{eq:HactTheta} as well as
\begin{equation}
  \label{eq:LT*}
  L_{\tau_1} \Theta^*_L(\tau, z, h) = \Theta_L(\tau, z, h)
\end{equation}
by \eqref{eq:Lphi*}.
When $V_{\sigma_1}$ is negative definite, we have $q_1 = n+2$ and the function $\theta^*_L(\tau, h) := \Theta^*_L(\tau, z, h)$ can be written explicitly as
\begin{equation}
  \label{eq:theta*}
  \begin{split}
      \theta^*_L(\tau, h) &=
      -
  \sum_{\lambda \in h(L') \cap V,~ \lambda \neq 0} \ebf(\tr(Q(\lambda)\tau))
  (4\pi |Q(\lambda)_1|)^{-n/2}\Gamma(n/2, 4\pi|Q(\lambda)_1|v_1) \phi_\lambda\\
  & +    \mathrm{CT}_{s = 0} \frac{v_1^{s + n/2 }}{s + n/2 }  \phi_0.
  \end{split}
\end{equation}
Although $\theta^*_L$ is not modular in $\tau$, difference of such functions will become modular after adding suitable holomorphic functions (see Theorem \ref{thm:L1pre}).

\begin{rmk}
  \label{rmk:odddim}
  When $\dim_F V$ is odd, all the constructions above still hold by working with  metaplectic covers.
{As this is not needed for most of the applications, we refrain from introducing more notations and refer the readers to \cite{Br12}. }
\end{rmk}

\subsection{Eisenstein Series and Siegel-Weil Formula.}

Let
$B\subset \SL_2$ the standard Borel subgroup, and $I(s, \chi) = \otimes_v I(s, \chi_v) := \mathrm{Ind}^{\SL_2(\Ab_F)}_{ B(\Ab_F)} \chi |\cdot|^s$  the induced representation of $ G(\Ab_\Qb)$ for $\chi = \chi_V$ the quadratic Hecke character associated to an $F$-quadratic space $(V, Q)$ with signature $((p_1, q_1), (n+2,0), \dots, (n+2,0))$.
For a standard section $\Phi 
\in I(s, \chi) $, 
the Eisenstein series
\begin{equation}
  \label{eq:Eis}
  E( g, s, \Phi) := \sum_{\gamma \in B(F) \backslash \SL_2(F)} \Phi(\gamma  g, s)
\end{equation}
converges absolutely for $\Re(s) \gg 0$ and has meromorphic continuation to $s \in \Cb$.
When $s_0 := n/2$, the map $\lambda: \Sc(V(\Ab_F)) \to I(s_0, \chi)$ defined by
\begin{equation}
  \label{eq:SWsec}
\lambda(\phi)( g) = (\omega( g)\phi)(0)  
\end{equation}
is ${\SL}_2(\Ab_F)$-equivariant, and $\lambda(\phi)$ can be extended uniquely to all $s$ and produce a standard sections in $I(s, \chi)$.

The infinite part $\otimes_{1 \le j \le d} I(s,  \chi_{\sigma_j})$ is generated by functions $\Phi^\kappa_{\infty} := \prod_j \Phi_\Rb^{\kappa_j}$ with $\kappa = (\kappa_j) \in \Zb^d$ satisfying $\kappa_j \equiv n/2 + 1\bmod{2}$, where
$$
\Phi_\Rb([k_\theta, 1]) = e^{i \theta /2},~ - \pi < \theta \leq \pi
$$
is the image of $\phi^{q + 2, q}_\infty$ defined in \eqref{eq:varphipq} under $\lambda$ for any $q \in \Nb$ \cite[Lemma 4.1]{BY11}. 
For $\phi_f \in \Sc(V(\hat F))$ and $\tka = \tka(V)$ as in \eqref{eq:kaV},
the Eisenstein series
\begin{equation}
  \label{eq:Eis1}
  E(\vtau, \phi_f, \tka) := v^{- \tka/2} E( g_\vtau, s_0, \lambda(\phi_f) \otimes \Phi^{ \tka}_\infty)
\end{equation}
is a Hilbert modular form of weight $\tka$.
For any even lattice $L \subset V$, we have
\begin{equation}
  \label{eq:EL}
  E_L(\vtau, \tka) := \sum_{\mu \in L'/L} E(\vtau, \phi_\mu, \tka) \phi_\mu \in \Ac_{ \tka, \rho_L}.
\end{equation}

Suppose $q = 2$. 
For any compact open $K \subset H(\hat \Qb)$ stabilizing $\hat L$ and acting trivially on  $\hat L'/ \hat L$, we have the Siegel-Weil formula \cite[Lemma 4.3]{BY11}
\begin{equation}
  \label{eq:SW}
  E_L(\vtau, \tka) = \frac{1}{\mathrm{vol}(X_K)} \int_{X_K} \Theta_L(\vtau, z, h) \Omega^n,~ n \ge 1,
\end{equation}
where $\Omega$ is the K\"ahler form on $X_K$ normalized as in \cite{Br12} and $\mathrm{vol}(X_K)  = \int_{X_K} \Omega^n$. 
For $n = 0$, the Siegel-Weil formula yields
\begin{equation}
  \label{eq:SW0}
  E_L(\vtau, \tka) = \frac{2}{\mathrm{vol}(X_K)} \sum_{[z, h] \in X_K} \Theta_L(\vtau, z, h). 
\end{equation}

When $q = 0$, we have $\tka = \vv{s_0 + 1}$ and the lattice $P = L$ is totally positive definite.
The classical Siegel-Weil formula yields
\begin{equation}
  \label{eq:SWdef}
  \begin{split}
      E_P(\tau) &:=  E_P(\tau, \tka) = c_K^{-1} \sum_{h \in H(\Qb)\backslash H(\hat\Qb)/K} \theta_P(\tau, h),\\
  c_K &= \sum_{h \in H(\Qb)\backslash H(\hat\Qb)/K} 1 \in \Nb.
  \end{split}
\end{equation}
Though $\theta_{P^\ell} = \theta_P^\ell$, the Eisenstein series $E_{P^\ell}$ is almost never the same as $E_P^\ell$!

Furthermore suppose  that $P \in \Uc_F^+$ with rank $2\rf$. Then $E_P(\tau)$  is in $\Mc_F^\theta$ and coincides with the Hecke Eisenstein series for $F$ of parallel weight  $\rf = n/2 + 1$. 
It can be written as (see section 2 in \cite{BL23})
\begin{equation}
  \label{eq:Eisenseries}
  E_P(\tau) =
  \sum_{\beta \in \Pb^1(F)} c(\beta, \tau)^{-\rf},~
  c(\beta, \tau) :=
  \begin{cases}
    A_\beta\cdot \Nm(\tau + \beta)& \beta \in F,\\
    1 & \beta = \infty,
  \end{cases}
\end{equation}
for certain non-zero integers $A_\beta$ depending only on $\beta$. 
Note that since $P^\ell \in \Uc^+_F$ for any $\ell \in \Nb$, we have $E_{P^\ell}(\tau) \in \Mc^\theta_F$.
Applying Lemma \ref{lemma:seriesvanish}, we can deduce the following results about zeros of theta functions. 
\begin{lemma}
  \label{lemma:supp}
  Let $P \in \Uc_F^+$ and $S \subset \Hb^d$ a finite set of points.
  Then there exists $\ell \ge 1$ such that $E_{P^\ell}$ does not vanish on $S$.
\end{lemma}
\begin{proof}
  Since $\SL_2(\Rb)^d$ acts transitively on $\Hb^d$, we can write $S = \{\tau_0, g_1 \tau_0,\dots, g_N \tau_0\}$ with $\tau_0 \in \Hb^d$ and $g_i \in \SL_2(\Rb)^d$, where we set $g_0$ to be the identity.
  Assume that for every $\ell \ge 1$, there exists $0 \le i \le N$ such that $E_{P^\ell}(g_i \tau_0) = 0$.
  Then the function
  $$
  f_\ell(\tau) := \prod_{i = 0}^N E_{P^\ell}(g_i \tau)
  $$
  vanishes at $\tau = \tau_0$ for all $\ell \ge 1$. 
  Using the expression \eqref{eq:Eisenseries} we can write
  $$
  f_\ell(\tau) =
  \sum_{\beta_0, \dots, \beta_N \in \Pb^1(F)}
\lp  \prod_{i = 0}^N c(\beta_i, g_i \tau) \rp^{-\rf \ell}.
  $$
  By Remark \ref{rmk:Freitag}, we have $  \prod_{i = 0}^N c(\beta_i, g_i \tau_0) = 0$ for all $\beta_1, \dots, \beta_N \in \Pb^1(F)$, which is clearly a contradiction since  $  \prod_{i = 0}^N c(\infty, g_i \tau_0) = 1$.
  \end{proof}

  \begin{lemma}
    \label{lemma:theta0}
    For any $\tau_0 \in \Hb^d$, there exist $P_1, P_2 \in \Uc_F^+$ in the same genus such that $\theta_{P_1}(\tau_0) \neq \theta_{P_2}(\tau_0)$. 
  \end{lemma}

  \begin{proof}
    Assume otherwise. Then for any $P \in \Uc_F^+$, the function $\theta_{P}(\tau, h)$ takes the same value at $\tau_0$ for all $h \in H(\hat\Qb)$ and
    we have
    $$
    E_P(\tau_0) = c_K \sum_{h \in H(\Qb)\backslash H(\hat\Qb)/K} \theta_P(\tau_0, h)
    =
    \theta_P(\tau_0)  c_K \sum_{h \in H(\Qb)\backslash H(\hat\Qb)/K} 1 = \theta_P(\tau_0).
    $$
    Therefore $E_{P^\ell}(\tau_0) = \theta_{P^\ell}(\tau_0) = \theta_{P}(\tau_0)^\ell = E_P(\tau_0)^\ell$ for all $\ell \ge 1$ and $P \in \Uc_F^+$.
    Using expression \eqref{eq:Eisenseries}, we obtain
    $$
    \sum_{\beta \in \Pb^1(F)} c(\beta, \tau_0)^{-\rf \ell}
    = E_P(\tau_0)^\ell,
    $$
which contradicts Corollary \ref{cor:bijection} since $c(\beta, \tau_0) \neq 0$ for all $\beta \in \Pb^1(F)$.
  \end{proof}

\begin{prop}[Partition of Unity]
  \label{prop:POU}
  Let $F$ be a totally real field. Then there exist $m \in \Nb$, $\tP_i, P_{1, i}, P_{2, i} \in \Uc^+_F$ with ranks $2\trf_i$ and $2\rf_i$ for $1 \le i \le m$
  such that $P_{1, i}$ and $P_{2, i}$ are in the same genus for all $i$ and there is no $\tau_0 \in \Hb$ such that $(E_{\tP_i} (\theta_{P_{1, i}}  - \theta_{P_{2, i}}))^\Delta  (\tau_0) = 0$ for all $i$.

Furthermore, for any $d_1, \dots, d_m, e_1, \dots, e_m \in \Nb$, there exists elliptic  modular forms $g_i \in M^!_{-(\trf_i d_i + \rf_{i} e_i) d}$ with rational Fourier coefficients such
  \begin{equation}
    \label{eq:POU}
    \sum_{i = 1}^m g_i( E^{d_i}_{\tP_i} \cdot (\theta_{P_{1, i}} - \theta_{P_{2, i}})^{e_i})^\Delta = 1.
  \end{equation}
\end{prop}

\begin{proof}
  Start with a point $\tau_0 \in \Hb$, we can find $P_{1}, P_{2} \in \Uc^+_F$ in the same genus satisfying $(\theta_{P_{1}} - \theta_{P_{2}})^\Delta(\tau_0) \neq 0$ by  Lemma \ref{lemma:theta0}.
  If $(\theta_{P_{1}} - \theta_{P_{2}})^\Delta$ has no zero on $\Hb$, then we can apply Lemma \ref{lemma:supp} to take $\tP_1, \tP_2 \in \Uc^+_F$ such that $E_{\tP_1}^\Delta$ and $E_{\tP_2}^\Delta$ do not have common zero on $\Hb$.
  The forms $E_{\tP_1} \cdot (\theta_{P_{1}} - \theta_{P_{2}})$ and 
$  E_{\tP_2} \cdot (\theta_{P_{1}} - \theta_{P_{2}})$ satisfy the first claim.

Otherwise, let $\tau_2, \dots, \tau_m \in \Hb$ be the zeros of $(\theta_{P_{1}} - \theta_{P_{2}})^\Delta$ in a fundamental domain $\Fc$ of $\SL_2(\Zb) \backslash \Hb$.
  By Lemma \ref{lemma:theta0}, there exists $P_{1, i}$ and $P_{2, i}$
  such that $\theta^\Delta_{P_{1, i}}(\tau_i) \neq \theta^\Delta_{P_{2, i}}(\tau_i)$
  for all $2 \le i \le m$.
  Let $S \subset \Fc$ be the finite set of the zeros of $\prod_{1 \le i \le m} (\theta_{P_{1, i}}^\Delta - \theta_{P_{2, i}}^\Delta)$. 
  Applying Lemma \ref{lemma:supp}, we can find $\tP \in \Uc^+_F$ such that $E_{\tP}$ does not vanish on $S$.
  Let $S' \subset \Fc$ be the finite set of the zeros of $E_{\tP}^\Delta\prod_{1 \le i \le m} (\theta_{P_{1, i}}^\Delta - \theta_{P_{2, i}}^\Delta)$.
    We can apply Lemma \ref{lemma:supp} to find $E_{\tP_i}$ for $2 \le i \le m$ such that they do not vanish on $S'$.
    Now the forms $E_{\tP} \cdot (\theta_{P_{1}} - \theta_{P_{2}})$ and $E_{\tP_i} \cdot  (\theta_{P_{1, i}} - \theta_{P_{2, i}})$ with $2 \le i \le m$ satisfy the first claim.

  To prove the second claim, we can write
$$
\Delta^{- ( \trf_i d_i +  \rf_i e_i) d} (E^{12d_i}_{\tP_i} (\theta_{P_{1, i}} - \theta_{P_{2, i}})^{12e_i})^\Delta = A_i(j) \in M^!_0 = \Qb[j]
$$
with $j = j(\tau)$ the $j$-invariant and $A_i(x) \in \Qb[x]$.
The first claim implies $\gcd(A_1, \dots, A_m) = 1$, i.e.\ there exists $B_1,  \dots, B_m \in \Qb[x]$ such that
$$
\sum_{i = 1}^m B_i (x) A_i(x)  = 1.
$$
Setting $g_i := B_i(j) \Delta^{-(\trf_i d_i + \rf_i e_i) d}
(E^{11d_i}_{\tP_i} (\theta_{P_{1, i}} - \theta_{P_{2, i}})^{11e_i})^\Delta$
proves the second claim.
\end{proof}

\subsection{Higher Green Function}
We follow \cite{BEY21} to recall higher Green function on the Shimura variety $X_K$ for $F = \Qb$.
Let $V/\Qb$ be a quadratic space of signature $(n, 2)$ and $L \subset V$ an even lattice with $K$ an open compact stabilizing $\hat L$.
Also, we denote
\begin{equation}
  \label{eq:rho0}
  \sigma_0 := \frac{n}{4} - \frac{1}{2}.
\end{equation}
For $\mu \in L'/L$ and $m \in \Zb + Q(\mu)$, the automorphic Green function is defined by
\begin{equation}
  \label{eq:Phimmu}
  \Phi_{m, \mu}(z, h, s) := 2 \frac{\Gamma(s + \sigma_0 )}{\Gamma(2s)}
\sum_{ \lambda \in h(L_{m, \mu})}
\lp \frac{m}{Q(\lambda_{z^\perp})} \rp^{s + \sigma_0}
F\lp s + \sigma_0 , s - \sigma_0, 2s; \frac{m}{Q(\lambda_{z^\perp})} \rp,
\end{equation}
where $F(a, b, c; z)$ is the Gauss hypergeometric function \cite[Chapter 15]{AS64}.
The sum converges normally on $X_K \backslash Z(m, \mu)$ for $s > \sigma_0 + 1$ and defines an eigenfunction of the Laplacian $\Delta$ on $\Db$, normalized as in \cite{Br02}, with eigenvalue $\frac{1}{2} (s - \sigma_0 - 1)(s + \sigma_0)$. 
Furthermore, it has a meromorphic continuation to $s \in \Cb$ with a simple pole at $s = \sigma_0 + 1$, whose constant term is denoted by $\Phi_{m,\mu}(z, h, \sigma_0 + 1)$ and the regularized theta lift of Hejhal-Poincar\'e series of index $(m, \mu)$ \cite{Br02}.

At $s = \sigma_0 + 1 + r$ with $r \in \Nb$, the function $\Phi_{m, \mu}(z, h, s)$ is called a \textit{higher Green function}.
For the unimodular lattice $L = M_2(\Zb)$ and $z = Z(z_1, z_2)$ as in Example \ref{exmp:M2Q}, we have
\begin{equation}
  \label{eq:GPhi}
  \Phi_m(z, 1, r+1) = \frac{2(-1)^{r}}{\Gamma(2r+2)} G_{r+1}^m(z_1, z_2),
\end{equation}
where $G_{s}^m$ is defined in \eqref{eq:Grf}. 
For a harmonic Maass form $f = \sum_{m, \mu} c(m, \mu) q^{-m} \phi_\mu + O(1)  \in H_{k - 2r, \bar \rho_L}$ with $k:= -2\sigma_0$, define
\begin{equation}
  \label{eq:Phijf}
  \Phi_L^r(z, h, f) := r! \sum_{m > 0,~ \mu \in L'/L} c(m, \mu) m^r \Phi_{m, \mu} (z, h, \sigma_0 + 1 + r)
\end{equation}
to be the associated higher Green function.
Following from the work of Borcherds \cite{Borcherds98} and generalization by Bruinier \cite{Bruinier02} (also see \cite{Via11, BEY21}), the function $\Phi_L^r$ has the following integral representation
\begin{equation}
  \label{eq:Phijint}
  \Phi_L^r(z, h, f) = 
 \lim_{T \to \infty}  \int_{\Fc_T}^{} \langle \tR^r_\tau f(\tau),  \overline{\Theta_L(\tau, z, h)} \rangle d \mu(\tau)
=  (-1)^r \lim_{T \to \infty}  \int_{\Fc_T}^{} \langle f(\tau), \overline{\tR^r_\tau \Theta_L(\tau, z, h)} \rangle d \mu(\tau),
\end{equation}
where $\Fc_T$ is the truncated fundamental domain of $\Gamma_\Qb \backslash \Hb$ at height $T> 1$.
\section{Real-analytic Hilbert Modular Forms and Algebraicity of Pairing}
\label{sec:mixmock}

In this section, we will prove the existence of certain real-analytic Hilbert modular forms by generalizing the proof of Theorem 3.7 in \cite{BF04},
and give some results concerning their Fourier coefficients.
The notations $F, D, \Oc, \df, d$ are the same as in section \ref{sec:prelim}.
\subsection{Certain real-analytic Hilbert modular forms}
\label{subsec:exist}
Let $\rho = \rho_L$ be a Weil representation,  $\kappa = (k_1, \dots, k_d) \in \Zb^d$,  $\Gamma \subset \Gamma_F$ a congruence subgroup, and $X = X(\Gamma) = \Gamma \backslash \Hb^d$ be the open Hilbert modular variety. 
By adding  finitely many cusps to $X$, we obtain the Baily-Borel compactification $X^{\mathrm{BB}}$.
It can also be constructed as the Proj of the ring of holomorphic modular forms on $X$, and is a normal, Noetherian scheme over $\Cb$.
When $\Gamma$ is neat, we fix a smooth toroidal compactification  $\tX$ of $X$.
It is a compact complex manifold, and a desingularization of $X^{\mathrm{BB}}$. 
We also have the natural map $\pi: \tX \to X^{\mathrm{BB}}$   and let
$\E$ be the boundary divisor on $\tX$.
Suppose $\tX$ is associated to a projective $\Gamma$-admissible decomposition (see \cite[section II.2]{AMRT}).

Denote  $\Oc$ and $\Ec^{p, q}$ the sheaf of holomorphic functions and smooth differential forms of type $(p, q)$ on $\tX$ respectively, and take the subsheaf $\Ec' := \ker(\Ec^{0, 1} \stackrel{\bar\partial}{\to } \Ec^{0, 2})$. 
Then the Dolbeault resolution of $\Oc$ gives us the short exact sequence
\begin{equation}
  \label{eq:ses}
0 \to \Oc \to
\Ec^{0, 0} \stackrel{\bar\partial}{\to}  \Ec' \to 0.
\end{equation}
For a Cartier divisor $\D$ and quasi-coherent sheaf $\Fc$ on $\tX$, we write $\Fc(\D)$ for the  corresponding twisting sheaf.
Also, let $\Lc_{\kappa, \rho}$ be the sheaf of modular forms of weight $\kappa$ and representation $\rho$ on $X$.
It extends to $X^{\mathrm{BB}}$ and $\tX$ by Koecher's principle, and we use $\Lc_{\kappa, \rho}  $ and $\tLc_{\kappa, \rho}$ to denote these extensions.
In particular,
\begin{equation}
  \label{eq:H0X}
  H^0(\tX, \Ec^{0, 0} \otimes \tLc_{\kappa, \rho}) \supset \Ac_{\kappa, \rho}(\Gamma).
\end{equation}
Note that $\pi_*\tLc_{\kappa, \rho} = \Lc_{\kappa, \rho}, \pi^*\Lc_{\kappa, \rho} = \tLc_{\kappa, \rho}$. 
When $\rho$ is trivial and $\kappa = (1, \dots, 1)$,
$\Lc = \Lc_{\kappa, \rho}$  is the determinant of the Hodge bundle and ample on $X^{\mathrm{BB}}$.
However, the extension $\tLc$ is trivial at the fiber of a cusp, and in general not ample on $\tX$. 
Nevertheless, we can use it along with twisting by $\E$ to prove the following result.

\begin{thm}
  \label{thm:vanish}
  In the notations above, for any $N \in \Nb$, 
   there exists $n_0, k \in \Nb_{> N}$ such that the following map
  \begin{equation}
    \label{eq:surj}
 H^0(\tX, \Ec^{0, 0} \otimes \tLc_{\kappa, \rho} \otimes  \tLc_{ \vv{k}}(- \E)^{\otimes n})  \stackrel{\bar \partial \otimes 1 \otimes 1}{\to}
 H^0(\tX, \Ec' \otimes \tLc_{\kappa, \rho} \otimes   \tLc_{\vv{k}}(-\E)^{\otimes n}  )
  \end{equation}
is surjective   for all $n \ge n_0$.
\end{thm}

\begin{proof}
  For simplicity, suppose $X^{\mathrm{BB}}$ has only one cusp $x$. 
  By Theorem 2.2 in Chapter IV of \cite{AMRT}), $\tX$ is the normalization of the blowing-up of $X^{\mathrm{BB}}$ at certain coherent sheaf $\mathcal{I}$ of ideals concentrated at $x$. 
  As $X^{\mathrm{BB}}$ is Noetherian, so is $\tX$ and $\pi$ is quasi-compact.
  We claim that $\tLc_{\vv{k}}(-\E) = \Oc_{\tX}(-\E) \otimes \pi^* \Lc^{\otimes k} $ is ample on $\tX$ for some $k > N$.

  Since normalization is a finite morphism in this case (see Lemma 33.27.1 in \cite[\href{https://stacks.math.columbia.edu/tag/0BXQ}{Tag 0BXQ}]{stacks}), it preserves ampleness and we can suppose that $\tX$ is the blowing-up of $X^{\mathrm{BB}}$. 
  By the discussion and Lemma 31.32.4(3) in \cite[\href{https://stacks.math.columbia.edu/tag/01OF}{Tag 01OF}]{stacks} and   Lemma 29.38.2 in  \cite[\href{https://stacks.math.columbia.edu/tag/01VL}{Tag 01VL}]{stacks}, we know that $\Oc_{\tX}(-\E) = \Oc_{\tX}(1)$ is $\pi$-relatively ample.
  Since $\Lc$ is ample on $X^{\mathrm{BB}}$, Lemma 29.37.7 in \cite[\href{https://stacks.math.columbia.edu/tag/01VG}{Tag 01VG}]{stacks} proves the claim.

  By considering the long exact sequence in cohomology associated to \eqref{eq:ses}, we see that the surjectivity of \eqref{eq:surj} is equivalent to the vanishing of $H^1(\tX, \tLc_{\kappa, \rho} \otimes \tLc_{\vv{k}}(- \E)^{\otimes n})$ for all $n$ sufficiently large. This follows from  standard vanishing result for cohomology (e.g. Lemma 30.17.1 in \cite[\href{https://stacks.math.columbia.edu/tag/01XO}{Tag 01XO}]{stacks}), which finishes the proof.
\end{proof}

\begin{prop}
  \label{prop:exist}
  For $\kappa \in \Zb^d$, let $\tka, \hka \in \Zb^d$ be as in \eqref{eq:hatkappa}.
  Suppose $f \in \Ac_{\tka, \rho}(\Gamma)$ is holomorphic in $\tau_j$ for $2 \le j \le d$. 
  Given any $g \in S_{\kappa'}(\Gamma)$ of parallel weight, there exists $\ell_0 \in \Nb$ and functions
$\hat G_\ell \in \Ac_{\hat\kappa + \ell \kappa', \rho}(\Gamma)$ for all $\ell \ge \ell_0$ such that they are holomorphic in $\tau_j$ for $2 \le j \le d$,
\begin{equation}
  \label{eq:L1pre}
L_{\tau_1}
 \hat{G}_\ell = g^\ell \cdot  f,
\end{equation}
and $\hat G_{\ell'} = g^{\ell' - \ell} \hat G_\ell$ for all $\ell' \ge \ell \ge \ell_0$.
\end{prop}
\begin{proof}
  Suppose $\Gamma$ is neat and $g$ is non-zero. Let $\kappa' = \vv{k'}$ with $k' \in \Nb$ and fix some $N > k'$.
  Now choose   $n_0, k > N$ as in Theorem \ref{thm:vanish} and set
  $\ell_0 = n_0 k$.
  Given $f \in \Ac_{\tka, \rho}(\Gamma)$, the differential form $v_1^{-2}g(\tau)^{\ell_0} f(\tau) d \bar \tau_1$ is in the kernel of $\bar \partial \otimes 1 \otimes 1$ since $f$ is holomorphic in $\tau_2, \dots, \tau_d$.
Furthermore, it is orders of vanishing at the cusps are at least $\ell_0$ since $f$ is bounded near the cusps.
  Therefore, it is a global section of $\Ec' \otimes \tLc_{\hat \kappa, \rho} \otimes \tLc_{\kappa'}(- \E)^{\otimes \ell_0}$.
  Note that
  $$
  \tLc_{\kappa'}(- \E)^{\otimes \ell_0}
  =     \tLc_{\vv{k'k}}(- k\E)^{\otimes n_0}
  \subset \tLc_{\vv{k'k}}(- k'\E)^{\otimes n_0}
  =  \tLc_{\vv{k}}(- \E)^{\otimes n_0 k'}
  $$
since $k' < N < k$ and $\Oc_{\tX}(-k\E)$ is a subsheaf of $\Oc_{\tX}(-k'\E)$. 
By Theorem \ref{thm:vanish}, there exists $\hat G_{\ell_0} \in H^0(\tX, \Ec^{0, 0} \otimes \tLc_{\hka, \rho} \otimes \tLc_{\vv{k}}(-\E)^{n_0 k'})$ such that $\bar\partial  (-2i  \hat G_{\ell_0}) = v_1^{-2} g^{\ell_0} f d\bar \tau_1$.
As $f, g$ are real-analytic, so is $\hat G_{\ell_0}$. 
So for any $\ell \ge \ell_0$, the real-analytic modular form $\hat G_{\ell} := g^{\ell - \ell_0} \hat G_{\ell_0} \in \Ac_{\hat\kappa + \ell \kappa', \rho}$ 
is holomorphic in $\tau_2, \dots, \tau_d$ and satisfies
$$
L_{\tau_1} \hat G_\ell(\tau) =
-2i v_1^{2} \partial_{\bar \tau_1} \hat G_\ell(\tau)  =
-2i g(\tau)^{\ell - \ell_0} v_1^{2}\partial_{\bar \tau_1} \hat G_{\ell_0}(\tau) =
 g(\tau)^\ell f(\tau). 
$$
From the construction, the last condition is also satisfied.

Finally for any congruence subgroup $\Gamma \subset \Gamma_F$, there exists a neat, normal subgroup $\Gamma' \subset \Gamma$ of finite index. Averaging the function $\hat G_{\ell} \in \Ac_{\hat \kappa + \ell \kappa ', \rho}(\Gamma')$ constructed above over $\Gamma/\Gamma'$ then gives the desired function in level $\Gamma$.
\end{proof}

Now, we will apply this result to the case when $f$ is the special value of a theta kernel and $g$ is the holomorphic theta function for a positive definite lattice.

\begin{thm}
  \label{thm:L1pre}
  Let $W$ be an $F$-quadratic space of dimension $2$ with signature as in \eqref{eq:sign}, and $P_1, P_2 \in \Uc_F^+$ positive definite, $\Zb$-unimodular $\Oc$-lattices of ranks $2\rf$ and in the same genus. 
  For an $\Oc$-lattice $N \subset W$,
  there exists $\ell_0 \in \Nb$ and $\hat\delta(\tau) = \hat\delta(\tau; N, h,  P_1, P_2, \ell) \in \Ac_{\vv{ 1 + \ell \rf}, \rho_N}$   for all $\ell \ge \ell_0$ and $h \in H_W(\hat\Qb)$
  having the following properties.
  \begin{enumerate}
  \item 
    It is holomorphic in $\tau_2, \dots, \tau_d$ and has exponential decay near the cusps.
\item
  It satisfies
  \begin{equation}
    \label{eq:hTprop}
    \begin{split}
      L_{\tau_1} \hat\delta(\tau) &= (\theta_{P_2}(\tau) - \theta_{P_2}(\tau))^\ell \lp  \theta_N(\tau) -   \theta_N(\tau, h) \rp,
    \end{split}
  \end{equation}
for all $\tau \in \Hb^d, h \in H_W(\hat \Qb)$ and $\ell \ge \ell_0$.
\item
  We can write 
\begin{equation}
  \label{eq:holpart}
  \hat\delta(\tau) = \hat\delta^+(\tau) +
 (  \theta_{P_1}(\tau) - \theta_{P_2}(\tau))^\ell \lp \theta_N^*(\tau) - \theta_N^*(\tau, h)\rp,
\end{equation}
where $\theta^*_N(\tau, h)$ is defined in \eqref{eq:theta*}, and  $\hat\delta^+$ is holomorphic in $\tau$ and $\Gamma_\infty$-invariant with respect to $\rho_N$.

\item
  Given the Fourier expansions
  $$
  \hat\delta(\tau) = \sum_{m, \mu} \hat a_{m, \mu}(v_1)\ebf(\tr(m\tau))\phi_\mu,~ \hat\delta^+(\tau, h) = \sum_{m, \mu}  \hat a^+_{m, \mu}\ebf(\tr(m\tau))\phi_\mu,$$
  we have $\hat a_{m, \mu}^+ =0$ unless $m \gg 0$, and
\begin{equation}
  \label{eq:FClim}
  \lim_{v_1 \to \infty} \hat a_{m, \mu}(v_1) = \hat a^+_{m, \mu}
\end{equation}
for all $m \in F$ and $\mu \in L'/L$.
\item
  For any $r_1, \dots, r_d \in \Nb_0$, we have
  \begin{equation}
    \label{eq:diff}
\lim_{|v| \to \infty}    \partial_{\tau_1}^{r_1} \dots     \partial_{\tau_d}^{r_d} (\hat\delta - \hat\delta^+) = 0.
  \end{equation}
  \end{enumerate}
\end{thm}

\begin{rmk}
  \label{rmk:choice}
  Up to holomorphic cusp forms of parallel weight $\rf \ell + 1$, the holomorphic part $\hat\delta^+$ is uniquely determined by the conditions above.
  We will show later in Theorem \ref{thm:algebraicity} that certain rational linear combinations of the Fourier coefficients of $\hat\delta^+$ are logarithms of algebraic numbers.
\end{rmk}
\begin{proof}
  By the Siegel-Weil formula, $\theta_N(\tau)$ and $\theta_N(\tau, h)$ have the same constant terms at all cusps, and their difference decays rapidly towards all cusps (see e.g.\ Proposition 5.1 in \cite{BY11}).
  By the same reason, $\theta_{P_1} - \theta_{P_2}$ is a holomorphic cusp form, and  we can apply Proposition \ref{prop:exist} above to $f(\tau) = \theta_N(\tau) - \theta_N(\tau, h) \in \Ac_{\kappa, \rho}$ and $g =\theta_{P_1} - \theta_{P_2} \in S_{\vv{\rf}}$.  Note that $\ell_0$  a priori depends on $h$. Since $H_W(\Qb) \backslash H_W(\hat \Qb)/K$ is finite for any open compact $K \subset H_W(\hat\Qb)$, there are only finitely many possible $h$ that give rise to different $\theta_N(\tau) - \theta_N(\tau, h)$, with possibly different $\ell_0$'s. By taking the maximum removes its dependence on $h$.

  Notice that $\hat\delta(\tau) - (\theta_{P_1}(\tau) - \theta_{P_2}(\tau))^\ell (\theta_N^*(\tau) - \theta_N^*(\tau, h))$ is annihilated by $L_{\tau_1}$ by \eqref{eq:LT*} and \eqref{eq:hTprop}. Therefore, it is holomorphic in $\tau_1, \dots, \tau_d$. This proves part (3).
  Equation \eqref{eq:theta*} implies that $\theta^*_N(\tau) - \theta^*_N(\tau, h)$ decays exponentially as $v_1 \to \infty$.
  Therefore, the same holds for $\hat\delta$ and the holomorphic part $\hat\delta^+$, whose Fourier coefficients are then supported only on totally positive indices. Equations \eqref{eq:FClim} and \eqref{eq:diff} now follows directly from \eqref{eq:theta*}.
\end{proof}

\subsection{ Whittaker Forms}
Suppose $d \ge 2$ for this section. 
We follow \cite{Br12, BY11} to recall Whittaker forms and their
regularized theta lifts.
For an even $\Oc$-lattice  $(L, Q)$  with signature as in \eqref{eq:sign} and $n > 2$, denote
\begin{equation}
  \label{eq:s0}
 s_0 := n/2,
\end{equation}
$\tka = \tka(V) = (s_0 - 1, s_0 + 1, \dots, s_0 + 1)$ as in \eqref{eq:kaV}
and $\kappa = (1 - s_0, s_0 + 1, \dots, s_0 + 1) \in \Zb^d$.
Given
$\mu \in L'/L$ and totally positive $m \in \df^{-1} + Q(\mu)$, the function
\begin{equation}
  \label{eq:fmm}
  f^{}_{m, \mu} (\vtau) := \frac{\Nm(4\pi m)^{s_0}}{(4\pi m_1)^{s_0} \Gamma(s_0)^d} \lp \Gamma(s_0) - \Gamma(s_0, 4\pi m_1 v_1) \rp
  e^{4\pi m_1 v_1} \ebf(-\tr(m \bar \vtau)) \phi_\mu, 
\end{equation}
is called a harmonic Whittaker form of weight $\kappa$ in the sense of \cite{Br12}.
The space generated by such forms is denoted by $H_{\kappa, \bar\rho_L}$.
Given $f = \sum_{\mu \in L'/L,~ m \gg 0} c(m, \mu) f_{m, \mu}  \in H_{\kappa, \bar\rho_L}$, the Fourier polynomial
\begin{equation}
  \label{eq:Pc}
  \Pc(f) := \sum_{\mu \in L'/L} \sum_{\substack{m \in \df^{-1} + Q(\mu)\\ m \gg 0}} c(m, \mu) q^{-m} \phi_\mu 
\end{equation}
is called its principal part.
It is invariant under $\Gamma_\infty := \{\smat{1}{b}{}{1} \in \SL_2(\Oc)\}$ with respect to $\bar\rho_L$. 
Conversely, given any polynomial of the above form, there is a unique harmonic Whittaker form $f_\Pc \in H_{\kappa, \bar\rho_L}$ with this principal part.
Note that such polynomial only depends on the finite quadratic modular $L'/L$. 
We say that $f$ or $\Pc(f)$ is \textit{rational} if the polynomial $\Pc(f)$ has rational coefficients.

Let
$\hka$ be the dual weight of $\kappa$ as in \eqref{eq:hatkappa}. 
There is a natural surjection $\xi = \xi^{(1)}_{\kappa}: H_{\kappa, \bar\rho_L} \to S_{\hka, \rho_L}$
\begin{equation}
  \label{eq:xi}
  \xi(f) := \sum_{\gamma \in \Gamma_\infty\backslash \Gamma_F}
  v_1^{-s_0-1} \overline{L_{\tau_1} f(\tau)} \mid_{\kappa, \rho_L} \gamma,
\end{equation}
where the sum above converges absolutely as $n > 2$.
This induces a bilinear pairing between $g = \sum_{n, \nu} b(n, \nu) q^n \phi_\nu \in M_{\hka, \rho_L}$ and $f \in H_{\kappa, \bar\rho_L}$ given by
\begin{equation}
  \label{eq:pairing}
  \{g, f\} = \{g, \Pc(f)\} := (g, \xi(f))_{\mathrm{Pet}} = \mathrm{CT}(\langle \Pc(f), g \rangle ) = \sum_{\mu \in L'/L,~ m \gg 0} c(m, \mu) b(m, \mu). 
\end{equation}
A harmonic Whittaker form $f$ is called \textit{weakly holomorphic} if $\xi(f)$ vanishes identically, i.e.\ $\{g, f\} = 0$ for all $g \in S_{\hka, \rho_L}$.
We use $M^!_{\kappa, \rho_L} \subset H_{\kappa, \rho_L}$ to denote the subspace of such forms.
Using the last expression in \eqref{eq:pairing}, we can extend $\{,\}$ to formal Fourier series in the parameter $\ebf(\tr(mu))$. 
For a subfield $\Fc \subset \Cb$, let $M^!_{\kappa, \bar\rho_L}(\Fc)$ denote the subspace of $M^!_{\kappa, \bar\rho_L}$  with Fourier coefficients in $\Fc$. 
Then the fact that $M_{\hka, \rho_L}(\Fc) = M_{\hka, \rho_L}(\Qb) \otimes \Fc$ implies that
\begin{equation}
  \label{eq:Qbasis}
  M^!_{\kappa, \bar\rho_L}(\Fc) =   M^!_{\kappa, \bar\rho_L}(\Qb) \otimes \Fc.
\end{equation}
\subsection{Regularized Theta Lifts}
For each $f = f_\Pc = \sum_{m, \mu} c(m, \mu) f_{m, \mu} \in H_{\kappa, \bar\rho_L}$, Bruinier computed its regularized theta lift in \cite{Br12} and constructed an Arakelov Green function $\Phi(z, h, f)$ for the divisor
\begin{equation}
  \label{eq:Zf}
Z(f) = Z(\Pc) := \sum_{\mu \in L'/L,~ m \gg 0} c(m, \mu) Z(m, \phi_\mu)  
\end{equation}
on the Shimura variety $X_K$.
When $s_0 = n/2 > 1$, Corollary 5.3 of \cite{BY11} expressed this as
\begin{equation}
  \label{eq:Phi}
  \begin{split}
      \Phi_L(z, h,f) &=
      \frac{1}{\sqrt{D}} \int^{\mathrm{reg}}_{\Gamma_\infty \backslash \Hb^d} \langle f(\tau),
      \overline{      \Theta_L(\tau, z, h) - E_L(\tau, \tilde\kappa)} \rangle (v_2\dots v_d)^{s_0+1} 
      d\mu(\tau)\\
      &\quad + B(f) (\Gamma'(1) + s_0^{-1})      ,
  \end{split}
\end{equation}
where $\Theta_L\in \Ac_{\tilde\kappa, \rho_L}$ with the weight $\tka$ given \eqref{eq:hatkappa}, $B(f) = B(\Pc) := \{E_L(\tau, \tka), \Pc\} \in \Cb$ and
\begin{equation}
  \label{eq:reg}
   \int^{\mathrm{reg}}_{\Gamma_\infty \backslash \Hb^d} G(\tau) d\mu(\tau)
:=
 \int_{v \in (\Rb_{> 0})^d} \lp \int_{u \in \Oc \backslash \Rb^d} G(\tau) du \rp \frac{dv}{\Nm(v)^2}
\end{equation}
for any $\Gamma_\infty$-invariant function $G$ on $\Hb^d$ such that the integral converges.

The integral in  \eqref{eq:Phi} converges for $s_0 = n/2 > 1$.
(see Proposition 5.2 of \cite{BY11}).
 Furthermore, Equation \eqref{eq:HactTheta} implies that
 \begin{equation}
   \label{eq:Phih}
   \Phi_L(z, h (h')^{-1}, f) =    \Phi_{L_{h'}}( z,  h, \iota_{h'} \circ f)    
 \end{equation}
 for any $h, h' \in H(\hat\Qb)$, where the isomorphism $\iota_{h'}: S_L \to S_{L_{h'}}$ is defined in section \ref{subsec:unilat}.
 
 By Theorem 5.14 of \cite{Br12}, for any $(z_0, h_0) \in \Db \times H(\hat \Qb)$, the function
 \begin{equation}
   \label{eq:Phi2}
   \Phi_L(z, h, f) + \sum_{m, \mu} c(m, \mu) \sum_{\lambda \in h_0(L_{m, \mu}) \cap z_0^\perp} \log | Q(\lambda_z)|
 \end{equation}
 is real-analytic in a neighborhood of $(z_0, h_0)$.
 By inspecting its proof, we have the following consequence.
 \begin{lemma}
   \label{lemma:Phidiff}
   Suppose $n > 2$ and $h_0, h_0' \in H(\hat \Qb)$ satisfy
   \footnote{
     Via the diagonal embedding, we tacitly view elements in $V(\Qb)$ as in $V(\hat\Qb)$. The set of $z_0$ in $\Db$ such that $V \cap z_0^\perp$ is non-trivial has measure 0.} 
 \begin{equation}
   \label{eq:h0cond}
   h_0\mid_{V \cap z_0^\perp} = h_0'\mid_{V \cap z_0^\perp}.
 \end{equation}
 Then $\Phi_L(z, h_0, f) -\Phi_L(z, h_0', f)$ is real-analytic in an open neighborhood
 of $z_0 \in \Db$, where it is given by
 \begin{equation}
   \label{eq:Phidiff}
   \Phi_L(z, h_0, f) -    \Phi_L(z, h_0', f) =
   \frac{1}{\sqrt{D}}
   \int_{\Gamma_\infty \backslash \Hb^d}^{\mathrm{reg}} \langle f(\tau),
   \overline{\Theta_L(\tau, z, h_0)  - \Theta_L(\tau, z, h_0')}
     \rangle (v_2\dots v_d)^{s_0+1} d\mu(\tau)
 \end{equation}
 In particular, the integral above converges at $z = z_0$.
 \end{lemma}
 \begin{proof}
   Without loss of generality, we can suppose $f = f_{m, \mu}$.
      From \eqref{eq:h0cond}, we have
   $$
   h_0(L_{m, \mu}) \cap z_0^\perp =  h_0(L_{m, \mu} \cap z_0^\perp) = h_0'(L_{m, \mu} \cap z_0^\perp) = h_0'(L_{m, \mu}) \cap z_0^\perp.
   $$
   So the singularities of $\Phi_L(z, h, f)$ and $\Phi_L(z, h', f)$ agree for $(z, h)$ and $(z, h')$ in open neighborhoods  $\Nc_1 \times \Nc_2$ and $\Nc_1' \times \Nc_2'$   of $(z_0, h_0)$ and $(z_0, h_0')$ respectively.
   Then $\Phi_L(z, h_0, f) - \Phi_L(z, h_0', f)$ is real-analytic in the open neighborhood $\Nc := \Nc_1 \cap \Nc_2 \subset \Db$ of $z_0$, which could be made to satisfy $h_0(L_{m, \mu}) \cap z^\perp = \{0\}$ for all $z \in \Nc\backslash\{z_0\}$. 

   Then convergence of the integral at $z \not\in Z(m, \mu)$ follows directly from Proposition 5.2 of \cite{BY11}. For such $z$, we can apply the unfolding calculation in the proof of Theorem 5.3 of \cite{Br12} while evaluating at $s = s_0$. Then the following integrals are identically equal for any $z \in \Nc \backslash \{z_0\}$
   \begin{align*}
     &\int_{u \in \Oc \backslash \Rb^d}
     \langle f_{m, \mu}(\tau), \overline{\Theta_L(\tau, z, h_0)  - \Theta_L(\tau, z, h_0')}  \rangle (v_2\dots v_d)^{s_0+1} du     \\
     =&
     \int_{u \in \Oc \backslash \Rb^d}
     \langle f_{m, \mu}(\tau), \overline{\Theta_L'(\tau, z, h_0)  - \Theta_L'(\tau, z, h_0')}  \rangle (v_2\dots v_d)^{s_0+1} du     ,
   \end{align*}
   where $\Theta_L'(\tau, z, h) := \Theta_L(\tau, z, h) - \sum_{\lambda \in h(L_{m, \mu}) \cap z^\perp} \phi_\mu(h^{-1} \lambda) \phi_\infty(\tau, z, \lambda) \phi_\mu$.
From the second expression, we see that the integral in \eqref{eq:Phidiff} also converges for $z = z_0$. 
 \end{proof}

 To evaluate the regularized theta integral, one can apply Stokes' theorem when certain primitive exist.
 The following lemma distilled from Theorems 6.3 and 7.2 in \cite{BY11} will be helpful for this purpose.

\begin{lemma}
  \label{lemma:stokes}
  Suppose $n = 2s_0 > 2$ and   $f = \sum_{m \gg 0,~ \mu \in L'/L} c(m, \mu) f_{m, \mu}$ is a harmonic Whittaker form in $H_{\kappa, \bar\rho_L}$. 
  Let $\eta \in \Ac_{\tka, \rho_L}$ such that $   \int_{\Gamma_\infty \backslash \Hb^d}^{\mathrm{reg}} \langle f(\tau), \overline{\eta(\tau)}  \rangle (v_2\dots v_d)^{s_0+1} d\mu(\tau)      $ converges.
  Suppose there exists
  $$
  \hat\eta(\tau) = \sum_{m \in F,~ \mu \in L'/L} \hat b_{m, \mu}(v_1) \ebf(\tr(m \tau)) \phi_\mu \in \Ac_{\hka, \rho_L}
  $$
  such that  $  L_{\tau_1} \hat\eta(\tau) = -2iv_1^2 \partial_{\bar\tau_1}\hat\eta(\tau) =\eta(\tau) $ and the $\Gamma_F$-invariant function $  \|\hat \eta(\tau) \|_{\mathrm{Pet}}$ is bounded on $\Hb^d$.
  Then 
  \begin{equation}
    \label{eq:stokes}
    \int_{\Gamma_\infty \backslash \Hb^d}^{\mathrm{reg}} \langle f(\tau), \overline{\eta(\tau)}  \rangle (v_2\dots v_d)^{s_0+1} d\mu(\tau)    =
  -(\hat\eta, \xi(f))_{\mathrm{Pet}} +     \lim_{T \to \infty}
    \sum_{m, \mu} c(m, \mu) \hat b_{m, \mu}(T).
  \end{equation}
\end{lemma}

\begin{rmk}
  \label{rmk:limit}
  Since $\xi(f)$ is a cusp form, the integral defining the Petersson inner product exists, and so does the limit in $T$.
\end{rmk}
  \begin{proof}
    We will give the proof for $f = f_{m, \mu}$. The general case can be proved the same way.
    Note that $m$ is totally positive.
By applying $L_{\tau_1}$ to the Fourier expansion of $\hat\eta$, we can write
$$
\eta(\tau) = \sum_{m \in F,~ \mu \in L'/L}  b_{m, \mu}(v_1) \ebf(\tr(m \tau)) \phi_\mu,~
 \partial_{v_1} \hat b_{m, \mu}(v_1) = v_1^{-2} b_{m, \mu}(v_1).
$$
Substituting this into the left hand side of \eqref{eq:stokes} gives us
\begin{align*}
  \frac{1}{\sqrt{D}}  \int_{\Gamma_\infty \backslash \Hb^d}^{\mathrm{reg}}
  \langle f(\tau), \overline{\eta(\tau)}  \rangle (v_2\dots v_d)^{s_0+1} d\mu(\tau)
  & = \lim_{T \to \infty} \int^T_{T^{-1}} b_{m, \mu}(v_1) \frac{\Gamma(s_0) - \Gamma(s_0, 4\pi m_1 v_1)}{\Gamma(s_0)} \frac{dv_1}{v_1^2}.
\end{align*}
Using integration by parts, we can rewrite the integral as
\begin{align*}
  \int^T_{T^{-1}} b_{m, \mu}(v_1)
  & \frac{\Gamma(s_0) - \Gamma(s_0, 4\pi m_1 v_1)}{\Gamma(s_0)} \frac{dv_1}{v_1^2}
    =
-      \frac{(4\pi m_1)^{s_0}}{\Gamma(s_0)}
  \int^T_{T^{-1}} \hat b_{m, \mu}(v_1) v_1^{s_0 - 1}  e^{-4\pi m_1 v_1} dv_1\\    
&+     \hat b_{m, \mu}(T)  \frac{\Gamma(s_0) - \Gamma(s_0, 4\pi m_1 T)}{\Gamma(s_0)} 
    -     \hat b_{m, \mu}(T^{-1})  \frac{\Gamma(s_0) - \Gamma(s_0, 4\pi m_1 T^{-1})}{\Gamma(s_0)}
\end{align*}
As $T \to \infty$, the first term on the right hand side becomes $- (\hat\eta, \xi(f))_{\mathrm{Pet}}$ after unfolding, and the second term becomes $\lim_{T \to \infty} \hat b_{m, \mu}(T)$.
For the third term, the boundedness of $\langle \hat\eta, \overline{\hat\eta} \rangle \Nm(v)^{s_0+1}$ implies that
$$
\hat b_{m, \mu}(t) = O(t^{(-s_0-1)/2})\text{ as }t \to 0,
$$
while $\Gamma(s_0) - \Gamma(s_0, 4\pi m_1 t) = O(t^{s_0})$ as $t \to 0$. Since $s_0 > 1$, the third term vanishes as $T \to \infty$. This finishes the proof.
  \end{proof}

\subsection{Algebraicity Result}
Fix a quadratic space $W$ over $F$ having signature as in \eqref{eq:sign} with $n = 0$. 
Let $N \subset W$ be an even $\Oc$-lattice, $K_N \subset H_W(\hat \Qb) $ an open compact fixing $\hat N$ and acting trivially on $\hat N' / \hat N = S_N$. 
Now, let $\tP, P_1, P_2 \in \Uc_F^+$ be lattices such that $P_1, P_2$ are in the same genus.
Denote their ranks by $2\trf$ and $2\rf$ with $\trf, \rf \in \Nb$. 
Given integers $\ell \in \Nb$ and $0 \le i \le \ell$, we write
\begin{equation}
  \label{eq:tU}
  \begin{split}
    L_{i} &:= N \oplus \tP \oplus P_1^i \oplus P_2^{\ell - i},\\
    \tU &:= \tP \otimes F,~ U := P_1^\ell \otimes F = P_2^\ell \otimes F,\\
V &:= L_i \otimes F = \tU \oplus U \oplus W.
  \end{split}
\end{equation}
Note that $\rho_{L_{i}} = \rho_N$ since $\tP, P_1, P_2$ are $\Zb$-unimodular.
For $h_W \in H_W(\hat \Qb) \subset H_V(\hat \Qb)$, we have
\begin{equation}
  \label{eq:Thetasplit}
\Theta_{L_{i}}(\tau, z_0, h_W) =  \theta_N(\tau, h_W) \theta_{\tP}(\tau)\theta_{P_1}(\tau)^i \theta_{P_2}(\tau)^{\ell - i} \in \Ac_{\vv{\trf + \rf \ell}  + (-1, 1, \dots, 1), \rho_N}  
\end{equation}
for a CM point $(z_0, h_W) \in Z(W, 1) \subset X_{V, K}$ defined over $E_W^{\mathrm{ab}}$ for a CM field $E_W/F$, 
and $K \subset H_V(\Ab)$ an open compact fixing $L_{i}$ for all $0 \le i \le \ell$.
Note that
$z_0$ depends only on the rational splitting $V = \tU \oplus U \oplus W$.

Theorem \ref{thm:L1pre} shows that for all $\ell \in \Nb$ sufficiently large, there exists $\hat\delta(\tau; N, h_W, P_1, P_2, \ell) \in\Ac_{\vv{\rf\ell  + 1}, \rho_N} $ for every $h_W \in  H_W (\hat\Qb)$
such that
\begin{equation}
  \label{eq:L-hat-delta}
L_{\tau_1} \hat\delta(\tau)
= (\theta_{P_1}(\tau) - \theta_{P_2}(\tau))^\ell (\theta_N(\tau) - \theta_N(\tau, h_W))
\end{equation}
and $L_{\tau_j} \hat\delta = 0$ for $2 \le j \le d$.
Recall that $\hat\delta^+(\tau) $ is the holomorphic part of $\hat\delta(\tau)$ as in \eqref{eq:holpart}.
We can now apply Lemmas \ref{lemma:Phidiff} and \ref{lemma:stokes} to the function
\begin{equation}
  \label{eq:etaP}
  \hat\eta_{\tP}(\tau) 
  := \theta_{\tP}(\tau)\cdot \hat\delta(\tau),~ 
  \hat\eta_{\tP}^+(\tau) :=\theta_{\tP}(\tau)\cdot \hat\delta^+(\tau),
\end{equation}
and express linear combinations of the Fourier coefficients of $\hat\eta_{\tP}^+$ in terms of special values of the regularized theta lift $\Phi_{L_{i}}(z, h, f)$ for any harmonic Whittaker form $f \in H_{\kappa, \bar\rho_N}$.

\begin{prop}
  \label{prop:pairing}
  Let $L_{i}, h_W, \hat\delta, \hat\eta_{\tP}$ be as above.
  Then
  \begin{equation}
    \label{eq:pair2}
    \sum_{0 \le i \le \ell} (-1)^{\ell - i} \binom{\ell}{i}
\lp    \Phi_{L_{i}}(z_0, 1, f) -     \Phi_{L_{i}}(z_0, h_W, f) \rp
    = -(\hat\eta_{\tP}, \xi(f))_{\mathrm{Pet}} + \{\hat\eta_{\tP}^+, f\}
  \end{equation}
    for any harmonic Whittaker form $f \in H_{\kappa, \bar\rho_N}$ with $\kappa = \widetilde{\tka(V)}$.
\end{prop}

\begin{proof}
  It is easy to see that $h_0 = 1, h_0' = h_W$ satisfy the condition \eqref{eq:h0cond}.
  Since $\hat\delta$ has exponential decay near the cusps, the function $\|\hat\eta_{\tP}(\tau)\|_{\mathrm{Pet}}$ is bounded on $\Hb^d$.
  Furthermore, equation \eqref{eq:L-hat-delta} and the holomorphicity of $\theta_{\tP}$ implies that
  $$
  L_{\tau_1} \hat\eta_{\tP}(\tau) =
  L_{\tau_1}( \theta_{\tP}(\tau) \hat\delta_{}(\tau)) =
  \theta_{\tP}(\tau)   L_{\tau_1}( \hat\delta_{}(\tau)) =
  \theta_{\tP}(\tau)  (\theta_{P_1}(\tau) - \theta_{P_2}(\tau))^\ell (\theta_N(\tau) - \theta_N(\tau, h_W)).
  $$
  If we denote $\hat b_{m, \mu}(v_1)$ and $\hat b_{m, \mu}^+$ the Fourier coefficients of $\hat\eta_{\tP}$ and $\hat\eta_{\tP}^+$ respectively, then we have
  $$
\lim_{v_1 \to \infty} \hat b_{m, \mu}(v_1) = \hat b_{m, \mu}^+. 
$$
by Theorem \ref{thm:L1pre}.
  Therefore Lemmas \ref{lemma:Phidiff} and \ref{lemma:stokes} together with \eqref{eq:Thetasplit} imply
  \begin{align*}
    \sum_{0 \le i \le \ell}& (-1)^{\ell - i} \binom{\ell}{i}
    \lp\Phi_{L_i}(z_0, 1, f) -
     \Phi_{L_i}(z_0, h_W, f) \rp\\
&    =
      \int_{\Gamma_\infty \backslash \Hb^d}^{\mathrm{reg}}
    \sum_{0 \le i \le \ell} (-1)^{\ell - i} \binom{\ell}{i}
  \langle f(\tau), \overline{\Theta_{L_i}(\tau, z_0, 1)  - \Theta_{L_i}(\tau, z_0, h_W)}  \rangle (v_2\dots v_d)^{s_0+1} d\mu(\tau)\\
&    =
      \int_{\Gamma_\infty \backslash \Hb^d}^{\mathrm{reg}}
  \langle f(\tau), \overline{(\theta_{N}(\tau) - \theta_{N}(\tau, h_W)) \theta_{\tP}(\tau) ( \theta_{P_1}(\tau) - \theta_{P_2}(\tau))^\ell}  \rangle (v_2\dots v_d)^{s_0+1} d\mu(\tau)\\
    &= -(\hat\eta_{\tP}, \xi(f))_{\mathrm{Pet}} +
      \lim_{T \to \infty} \sum_{m, \mu} c(m, \mu) \hat a_{m, \mu}(T)
      = -(\hat\eta_{\tP}, \xi(f))_{\mathrm{Pet}} +       \{\hat\eta_{\tP}^+, f\} .
  \end{align*}
  This finishes the proof.
\end{proof}

\begin{rmk}
  \label{rmk:Ph}
  If we replace $\tP$ by $\tP_{\th}$ with $\th \in H_{\tU}(\hat\Qb)$, then $\tP_{\th}$ is in the same genus as $\tP$ and $N \oplus  \tP_{\th} \oplus P_1^i \oplus P_2^{\ell - i} = (\th^{-1} L_i) \cap V = L_{i, \th}$ with $\th \in H_{\tU}(\hat\Qb) \subset H_V(\hat\Qb)$. 
  Furthermore, $\iota_{\th} \circ f = f$ since $\tP$ is $\Zb$-unimodular.
  The proposition above and Equation \eqref{eq:HactTheta} imply that the function $\hat\eta_{\tP_{\th}}$ satisfies
    \begin{equation}
      \label{eq:pair2'}
    \sum_{0 \le i \le \ell} (-1)^{\ell - i} \binom{\ell}{i}
\lp    \Phi_{L_{i}}(z_0, (1, \th^{-1}), f) -     \Phi_{L_{i}}(z_0, (h_W, \th^{-1}), f) \rp
    = -(\hat\eta_{\tP_{\th}}, \xi(f))_{\mathrm{Pet}} + \{\hat\eta_{\tP_{\th}}^+, f\}
  \end{equation}
for all $\th \in H_{\tU}(\hat\Qb)$.
\end{rmk}
  
If $c(m, \mu) \in \Zb$ for all $m $ and $\mu$, then Theorem 6.8 of \cite{Br12} implies that there exists a meromorphic modular form $\Psi_L(z, h, f)$ on $X_K$ with weight $-B(f)$ and a finite order multiplier system such that
\begin{equation}
  \label{eq:Psif}
- \log \|\Psi_L(z, h, f)\|^2_{\mathrm{Pet}} = \Phi_L(z, h, f).   
\end{equation}
Furthermore, the divisor of $\Psi_L(z, h, f)$ is the special cycle $Z(f)$, which is defined over $F$.
Let $M \in \Nb$ be the order of the multiplier system.
Then up to a locally constant function $C_{L}(z, h, f)$ on $X_K$, 
the form $\Psi_L(z, h, f)^M$ equals to a meromorphic modular form $R_{L}(z, h, f)$ of weight $-B(M \cdot f) \in \Zb$ on $X_K$ defined over $\sigma_1(F)$, and we can write
 \begin{equation}
   \label{eq:PhiR}
\Phi_L(z, h, f) = 2 \log |C_{L}(z, h, f)| - \frac{2}{M} \log \| R_{L}(z, h, f)\|_{\mathrm{Pet}}.    
\end{equation}

Now we let $L = L_i$ for $0 \le i \le \ell$ as in Proposition \ref{prop:pairing}.
Since the spinor norm $\nu: H_{\tU} \to T$ is surjective, the subgroup $H_{\tU}(\hat\Qb) \subset H_V(\hat\Qb)$ acts transitively on the connected components of $X_K$ by \eqref{eq:pi0}.
Therefore for any open compact $K_{\tU} \subset K \cap H_{\tU}(\hat\Qb)$, the quantity
\begin{equation}
  \label{eq:Cfind}
 \Cc_{i}(K_{\tU}, f) := \prod_{\th \in H_{\tU}(\Qb)_+ \backslash H_{\tU}(\hat\Qb)/K_{\tU}} C_{L_i}(z, (h_W, \th), f)
\end{equation}
is independent of $h_W \in H_W(\hat\Qb)$ and $z \in \Db$. 

Now, we are ready to state and prove the main result of this section.
\begin{thm}
\label{thm:algebraicity}
Let $N, h, P_1, P_2, \ell$ and $\hT(\tau; N, h, P_1, P_2, \ell) \in \Ac_{\vv{\rf \ell + 1}, \rho_L}$ be the same as in Theorem \ref{thm:L1pre} and $\tP, L_i, V, \kappa$ the same as in Proposition \ref{prop:pairing}.
%
%
%
For any weakly holomorphic Whittaker form $f = f_\Pc \in M^!_{\kappa, \bar\rho_N}$ with coefficients in a number field $\Fc \subset \Cb$, there exists $\lambda_1, \dots, \lambda_m \in \Fc$ and $\alpha_1, \dots, \alpha_m \in E_W^{\mathrm{ab}}$ independent of $h_W$ such that 
\begin{equation}
  \label{eq:algpair}
  \{ E_{\tP}(\tau) \hT^+(\tau), \Pc   \} = \sum_{i = 1}^m \lambda_i \log |\alpha_i/\sigma(\alpha_i)|,
\end{equation}
where
 $\sigma \in \Gal(E_W^{\mathrm{ab}}/E_W)$ is the  element associated to $h_W \in H_W(\hat\Qb) \cong \hat E_{W}^\times$ 
via class field theory.
\end{thm}

\begin{rmk}
  \label{rmk:m}
  We can choose $\lambda_i, \alpha_i$ above such that the index $m$ is bounded by the degree of $\Fc/\Qb$.
  Due to the order the multiplier system of the meromorphic modular form $\Psi(z, h, f)$ in \eqref{eq:Psif}, 
    the denominator of $\lambda_i$ depends on $f$ even when it has integral coefficients 
  \end{rmk}

\begin{proof}
  By \eqref{eq:Qbasis}, we can write $f = \sum_{i=1}^m \lambda_i f_i$ with $f_i\in M^!_{\kappa, \bar\rho_N}(\Qb)$ and suppose that $\Fc = \Qb$.
  By replacing $f$ with $M\cdot f$ for some $M \in \Nb$, we can suppose that $f$ has integral Fourier coefficients and the modular form $\Psi_{L_i}(z, h, f)$ in \eqref{eq:Psif} has trivial character for all $0 \le i \le \ell$.
  Let $R_L(z, h, f)$ and  $\Cc_L(z, h, f)$ be the same as in  \eqref{eq:PhiR}.

  Let $K \subset H_V(\hat \Qb)$ be an open compact fixing $\hat L_i$, acting trivially on $\hat L'_i/\hat L_i = \hat N' /\hat N$, and $K \cap H_W(\hat\Qb)$ contains $K_N$.
  Denote $K_{\tU} := K \cap H_{\tU}(\hat\Qb)$ an open compact in $H_{\tU}(\Qb)$ and $C_i(K_{\tU}, f)$ the same as in   \eqref{eq:Cfind}.
Fix $S_W \subset  H_W(\hat \Qb)$
coset representatives of  $H_W(\Qb)\backslash H_W(\hat \Qb)/K_N $.

  Since $f$ is weakly holomorphic, the cusp form $\xi(f)$ vanishes identically.
Combining the Siegel-Weil formula in \eqref{eq:SWdef}, Proposition\ \ref{prop:pairing}, Remark \ref{rmk:Ph} and Equations \eqref{eq:PhiR}, \eqref{eq:Cfind}, we can then write
\begin{align*}
&  \{ E_{\tP}(\tau) \hT^+(\tau), \Pc   \}
     =
    c_{K_{\tU}}^{-1}  \sum_{\th \in   H_{\tU}(\Qb)\backslash H_{\tU}(\hat\Qb) /K_{\tU}}
  \{ \theta_{\tP}(\tau, \th) \hT^+(\tau), \Pc   \}    \\
  &=
    c_{K_{\tU}}^{-1}  \sum_{\substack{\th \in   H_{\tU}(\Qb)\backslash H_{\tU}(\hat\Qb) /K_{\tU}\\ 0 \le i \le \ell}} (-1)^{\ell - i} \binom{\ell}{i}
\lp    \Phi_{L_{i}}(z_0, (1, \th), f) -     \Phi_{L_{i}}(z_0, (h_W, \th), f) \rp\\
&    = 
2     c_{K_{\tU}}^{-1}  
\sum_{0 \le i \le \ell} (-1)^{\ell - i} \binom{\ell}{i}
\lp \log \left|\frac{\Cc_i(K_{\tU}, f)}{ \Cc_i(K_{\tU}, f)}\right| -
\sum_{\th \in   H_{\tU}(\Qb)\backslash H_{\tU}(\hat\Qb) /K_{\tU}}
  \log \left\|\frac{ R_{L_i}(z_0, (1, \th),f)}{R_{L_i}(z_0, (h_W, \th), f)}\right\|_{\mathrm{Pet}}\rp\\
& =  -\frac{2}{    c_{K_{\tU}} \cdot \# S_W}  
\sum_{\th \in   H_{\tU}(\Qb)\backslash H_{\tU}(\hat\Qb) /K_{\tU}}
  \log |Q_{}(z_0, (1, \th))/Q_{}(z_0, (h_W, \th))|,\end{align*}
where $Q_{}(z, h)$  is a meromorphic function on $X_{K'}$ defined over $\sigma_1(F)$ given by
$$
Q_{}(z, h) := 
\prod_{0 \le i \le \ell}
\prod_{h_0 \in S_W}
\lp \frac{  R_{L_i}(z, h, f)}{R_{L_i}(z, h (h_0, 1), f)}\rp^{(-1)^i \binom{\ell}{i}}
  $$
  with $K' :=  K \cap_{h_0 \in S_W} h_0 K h_0^{-1}$.
  Since $H_W$ is abelian, we have  $K_N\times K_{\tU} \subset K'$.
  Therefore, the CM 0-cycle $Z(W)$ also lies on $X_{K'}$ and each CM point is defined over a number field $E_{K_N} \subset E_W^{\mathrm{ab}}$.
  The function $R_{L_i}$ is defined over $\sigma_1(F)$, as well as the natural map $X_{K'} \to X_K$ given by right multiplication with $h_0 \in H_W(\hat\Qb) \subset H_V(\hat\Qb)$ (see \cite[page 46]{Kudla97}).
  Therefore, the modular function $Q_{}(z, h)$ is also defined over $\sigma_1(F)$.
  Furthermore, it is  non-zero at $(z, h) = (z_0, (h_W, 1))$ for all $h_W \in H_W(\hat \Qb)$. These values are algebraic numbers satisfying
  \begin{align*}
    \sigma(Q_{}(z_0, (1, \th)))
      =     Q_{}(z_0, (h_W, \th))
  \end{align*}
  as $\sigma$ fixes the function $Q(z,h)$, which is defined over $\sigma_1(F)$, and acts on CM points in the CM 0-cycle $Z(W, \th)$ by Shimura's reciprocity law (see sections 3.1 and 5.3 in \cite{AGHMP18} for the relevant case here).
  Setting $\alpha = \prod_{\th \in   H_{\tU}(\Qb)\backslash H_{\tU}(\hat\Qb) /K_{\tU}}  Q_{}(z_0, (1, \th))$ and $\lambda =   -\frac{2}{    c_{K_{\tU}} \cdot \# S_W}   $  finishes the proof.
\end{proof}

\section{Proofs of Theorems}
\label{sec:pfs}
Now we are ready to prove the three theorems from the introduction.

  \subsection{Proof of Theorem \ref{thm:diff}}
  Let $N \subset W$ be a lattice such that $K \cap H_W(\hat \Zb)$ fixes $\hat N$ and acts trivially on $\hat N'/\hat N$. 
  By passing to sublattice, we can suppose that $L = \tilde L \oplus N_0$ with $N_0 = \Res_{F/\Qb}N$ and $\tilde L$ positive definite of rank $2 + n - 2d$. Then for any $(z_0, h) \in Z(W_\Qb)$, we have
$
\Theta_L(\tau, z_0, h) = \theta_{\tilde L}(\tau) \otimes \theta_N^\Delta(\tau, h)$.
By replacing $N$ with $h_1 \hat N \cap V$ if necessary, we can suppose that $h_1 = 1$ and write $h = h_2$.
We can apply Lemma \ref{lemma:RRC} to $g = \theta_{\tilde L}, k = 1 - n/2 - 2r, \ell = 1 + n/2 - d$ as $k + \ell = -d + 2 - 2r \le 1 - 2r$. 
This implies
\begin{align*}
  \langle \tR^r_\tau f(\tau),      \overline{\Theta_L(\tau, z_0, h)} \rangle
&=
\left\langle\langle \tR^r_\tau f(\tau),  \overline{\theta_{\tilde L}(\tau)}
\rangle_{N_0}, \overline{\theta^\Delta_N(\tau, h)} \right\rangle
  = \sum_{j = 0}^r c_{r, r, j}  \langle \tR^{r-j}_{\tau} f_j(\tau), \overline{\theta_N^\Delta(\tau, h)} \rangle,
\end{align*}
 where $c_{r, r, j} \in \Qb$ and 
$$
f_j(\tau) :=
\Cc_{(k, \ell), j} \langle f(\tau_1), \overline{\theta_{\tilde L}(\tau_2)}\rangle_{N_0} \in M^!_{2-2r-d+2j, \overline\rho_{N_0}}
$$
has rational Fourier coefficients.
From the integral representation \eqref{eq:Phijint}, we see that it suffices to prove the theorem with $f = f_0$ and $L = N_0$, in which case
$$
\langle f(\tau), \overline{\tR^r_\tau  \Theta_L(\tau, z_0, h)} \rangle
= \langle f_0(\tau),    \overline{\tR^r_\tau \theta_N^\Delta(\tau, h)} \rangle .
$$

Now for $1 \le i \le m$, let $\tP_i, P_{1, i}, P_{2, i} \in \Uc^+_F$ be $\Oc$-lattices of ranks $2\trf_i, 2\rf_i$ as in Proposition \ref{prop:POU}.
For any $\ell \in \Nb$, it gives us $g_i \in M^!_{-(\trf_i  +  \rf_i \ell_i)(r+1)d}$ such that
$$
\langle f_0(\tau),    \overline{\tR^r_\tau \theta_N^\Delta(\tau, h)} \rangle
=
\langle f_0(\tau),    \overline{\tR^r_\tau \theta_N^\Delta(\tau, h)} \rangle
\sum_{i = 1}^m g_i(\tau) (E_{\tP_i}^\Delta(\tau)(\theta_{P_{1, i}}^\Delta(\tau) - \theta_{P_{2, i}}^\Delta(\tau))^{\ell_i})^{r+1}  ,
$$
and it suffices to prove the theorem for each $i$.  
We drop the index $i$ and have
\begin{equation}
  \label{eq:diff0}
  \begin{split}
      \Phi^r_L&(z_0, 1, f) -     \Phi^r_L(z_0, h, f) \\
&=
    \lim_{T \to \infty}  \int_{\Fc_T}^{} \langle f_0(\tau) g(\tau),
  \overline{
  (G^\Delta(\tau))^{r+1}
  \tR^r_\tau \lp  \theta_N^\Delta(\tau) - \theta_N^\Delta(\tau, h) \rp} \rangle d \mu(\tau).
  \end{split}
\end{equation}
where we set $ G := E_{\tP} \cdot (\theta_{P_1} - \theta_{P_2})^\ell$.
Fix an $\ell \ge \ell_0$ as in Theorem \ref{thm:L1pre} and  we obtain a real-analytic Hilbert modular form $\hat\delta(\tau) = \hat\delta(\tau; N, h, P_1, P_2,{\ell}) \in \Ac_{\vv{1 + \rf \ell}, \rho_N}$ with the property
\begin{equation}
  \label{eq:L1heta}
  L_{\tau_1} \hat\eta(\tau) = G(\tau)(\theta_N(\tau) - \theta_N(\tau, h)),~
  \hat\eta := E_{\tP} \cdot \hat\delta \in \Ac_{\hka, \rho_N},  
\end{equation}
where $\kappa := (1-\trf -\rf\ell, \trf + \rf\ell + 1, \dots, \trf + \rf\ell+1)$ and $\hka$ is defined as in \eqref{eq:hatkappa}.
We can now use the differential operator $\Dc_{\vv{1}, r}$ from \eqref{eq:Dc} to define
 a real-analytic, elliptic modular form
\begin{equation}
  \label{eq:deltar}
  \hat\eta_r := \Dc_{\vv{1}, r} ( \hat\eta, G)
  \in \Ac_{k, \rho_{N_0}},~
  k:=    d ((1 + r)(\trf + \rf \ell) + 1) + 2r.
\end{equation}
By Lemma \ref{lemma:lower}, this function satisfies
\begin{align*}
  L_\tau \hat \eta_r
  &=
    G^\Delta(\tau)^{r+1} \tR^r (L_{\tau_1}\hat\eta/G)^\Delta
= G^\Delta(\tau)^{r+1}  \tR^r (\theta^\Delta_N(\tau) - \theta^\Delta_N(\tau, h)).
\end{align*}
We can now apply Stokes' theorem to the right hand side of \eqref{eq:diff0} and
part (5) of Theorem \ref{thm:L1pre} to obtain
\begin{align*}
  \Phi^r_L(z_0, 1, f) -     \Phi^r_L(z_0, h, f)
  &=
-    \lim_{T \to \infty}  \int_{0}^{1} \langle f_0(u+iT) g(u + iT), \overline{\hat\eta_r(u + iT) }\rangle du
= - \{ f_0 g, \hat\eta^+_r\},
\end{align*}
where $\hat\eta^+_r :=  \Dc_{\vv{1}, r} (\hat\eta^+, G)$ is a formal Fourier series with coefficients being $F$-linear combinations of those of $\hat\eta^+$.
Note there exists a harmonic Whittaker form $g_r \in H_{\kappa, \bar\rho_N}$ with
 Fourier coefficients in $F$ such that
$$
\{f_0g,
 \Dc_{\vv{1}, r} (\delta, G )\} = \{g_r, \delta\}.
$$
for any Fourier series $\delta(\tau) = \sum_{\mu \in S_N,~ m \in F} c(m, \mu) \ebf(\tr(\mu\tau)) \phi_\mu $ that is $\Gamma_\infty$-invariant with respect to $\rho_N$.
If $\delta \in S_{\hka, \rho_N}$, then $\Dc_{\vv{1}, r}(\delta, G)$ is in $S_{k, \rho_{N_0}}$ and this pairing vanishes as $f_0 g \in M^!_{2 -k, \overline\rho_{N_0}}$. Therefore $g_r$ is weakly holomorphic and  \eqref{eq:maindiff} follows from
$$
  \Phi^r_L(z_0, 1, f) -     \Phi^r_L(z_0, h, f)
  = - \{ f_0 g, \hat\eta^+_r\} = 
  - \{  g_r, \hat\eta^+\}
  =    - \{  g_r, E_{\tP}\hat\delta^+\} 
$$
and Theorem \ref{thm:algebraicity}.

When $d = 2 = n/2$, we already have $L = N_0$ and the reduction step in the first paragraph above is not necessary.
The function $\hat\eta_r^+$ is simply $(G^\Delta)^{r+1} \Cc_{(1, 1), r}(\hat\eta^+/G)$, and the last claim follows from Example \ref{exmp:wt11}.

\subsection{Proof of  Theorem \ref{thm:conj}}
As $Z(W_\Qb)$ is defined over $\Qb$, i.e.\ $Z(W_\Qb) = Z(W(2)_\Qb)$ on $X_{V, K}$, we can apply Theorem 5.10 in \cite{BEY21} to conclude that
$$
\frac{2}{\mathrm{deg}(Z(W_\Qb))} \Phi_L^r(Z(W_\Qb), f)
= \mathrm{CT}(\langle f, \Cc_{(1, 1), r}(\Ec_L^+)\rangle),
 $$
 where $\Ec_L^+$ is the holomorphic part of the derivative of an incoherent Eisenstein series, which is a real-analytic Hilbert modular form of weight $(1, 1)$.
 From Example \ref{exmp:wt11}, we see that the constant term of $\Cc_{(1, 1), r}(\Ec_L^+)$ vanishes when $r \ge 1$.
 Furthermore, since $f$ has rational Fourier coefficients, the term $\mathrm{CT}(\langle f, \Cc_{(1, 1), r}(\Ec_L^+)\rangle)$ is a rational linear combinations of the non-zero Fourier coefficients of $\Ec_L^+$, which are rational multiple of logarithms of integers by Proposition 4.6 in \cite{BKY12}.
 Therefore, we have
 \begin{equation}
   \label{eq:alg-all}
\sum_{(z_0, h) \in Z(W_\Qb)} \Phi_L^r(z_0, h, f) = \frac{1}{c'} \log |a|   
 \end{equation}
for some $c' \in \Nb$ and $a \in \Qb$. 

We can now apply Theorem \ref{thm:diff} to find  $c \in \Nb$ and $\alpha \in E' \subset E_W^{\mathrm{ab}}$ such that
$$
\Phi^r_L(z_0, 1, f) -     \Phi^r_L(z_0, h, f) = \frac{1}{c' c} \log \left| \frac{\alpha}{\sigma(\alpha)} \right|
$$
for all $(z_0, h) \in Z(W_\Qb)$. 
 Denote $N := |Z(W_\Qb)|/2 \in \Nb$ and $\beta := \prod_{\sigma \in \Gal(E'/E_W)} \sigma(\alpha) \in E_W$.
The system of equations above has the unique solution
$$
\Phi^r_L(z_0, h, f) = \frac{1}{2N c' c} \log | \sigma(\tilde \alpha)|
$$
with $\tilde{\alpha} = \alpha^{2N} a^c / \beta^2 \in E'$.
This finishes the proof.
\subsection{Proof of  Theorem   \ref{thm:GZ}}
When $r$ is even, this follows from the discussion in Example \ref{exmp:biquad} and Theorem \ref{thm:conj}. 
When $r$ is odd and $d_1$ is fundamental, Theorem 7.13 of \cite{BEY21} gives the algebraicity analogous to \eqref{eq:alg-all} with the left hand side replaced by certain partial average.
Using this and proceeding with the rest of the argument in the proof of Theorem \ref{thm:conj} gives Theorem \ref{thm:GZ}. 

\bibliography{CMHG}{}
\bibliographystyle{amsplain}

\end{document}